\documentclass[reqno]{amsart}
\usepackage{a4wide}
\usepackage[english]{babel}
\usepackage{enumerate}
\usepackage{csquotes}
\usepackage[inline,shortlabels]{enumitem}
\usepackage{xspace}

\usepackage[hyperfootnotes=false]{hyperref} 
\usepackage{color}
\hypersetup{
	colorlinks = true,
	linkcolor = {blue},
	urlcolor = {red},
	citecolor = {blue}
}

\makeatletter
\newcommand{\mysubsubsection}[1]{\subsubsection*{\bfseries #1}}
\renewcommand{\tocsection}[3]{
  \indentlabel{\@ifnotempty{#2}{\ignorespaces#1 #2\quad}}\bfseries#3}
\renewcommand{\tocsubsection}[3]{
  \indentlabel{\@ifnotempty{#2}{\ignorespaces#1 #2\quad}}#3}

\newcommand\@dotsep{4.5}
\def\@tocline#1#2#3#4#5#6#7{\relax
  \ifnum #1>\c@tocdepth
  \else
    \par \addpenalty\@secpenalty\addvspace{#2}
    \begingroup \hyphenpenalty\@M
    \@ifempty{#4}{
      \@tempdima\csname r@tocindent\number#1\endcsname\relax
    }{
      \@tempdima#4\relax
    }
    \parindent\z@ \leftskip#3\relax \advance\leftskip\@tempdima\relax
    \rightskip\@pnumwidth plus1em \parfillskip-\@pnumwidth
    #5\leavevmode\hskip-\@tempdima{#6}\nobreak
    \leaders\hbox{$\m@th\mkern \@dotsep mu\hbox{.}\mkern \@dotsep mu$}\hfill
    \nobreak
    \hbox to\@pnumwidth{\@tocpagenum{\ifnum#1=1\bfseries\fi#7}}\par
    \nobreak
    \endgroup
  \fi}
\AtBeginDocument{
\expandafter\renewcommand\csname r@tocindent0\endcsname{0pt}
}
\def\l@subsection{\@tocline{2}{0pt}{2.5pc}{5pc}{}}
\makeatother

\usepackage{amsmath}
\usepackage{amsthm}
\usepackage{amssymb}
\usepackage{mathrsfs} 
\usepackage{mathtools}
\usepackage{nicefrac}

\newcounter{results}[section]

\theoremstyle{plain}
\newtheorem{theorem}[results]{Theorem}
\newtheorem{lemma}[results]{Lemma}
\newtheorem{proposition}[results]{Proposition}
\newtheorem{corollary}[results]{Corollary}

\theoremstyle{remark}
\newtheorem{remark}[results]{Remark}
\newtheorem{example}[results]{Example}

\theoremstyle{plain}
\newtheorem{definition}[results]{Definition}

\numberwithin{equation}{section}

\newcommand{\eps}{\ensuremath{\varepsilon}} 
\newcommand{\nchi}{{\raise.3ex\hbox{$\chi$}}}

\newcommand{\rmC}{\mathrm C}
\newcommand{\proj}{\ensuremath{\mathbf{pr}}} 

\renewcommand{\exp}{\operatorname{\mathsf{exp}}}
\DeclareMathOperator*{\argmin}{arg\,min} 
\newcommand{\intt}[1]{\operatorname{int}\left(#1\right)}
\newcommand{\restr}[1]{\lower3pt\hbox{$|_{#1}$}}

\newcommand{\weakto}{\ensuremath{\rightharpoonup}} 
\newcommand{\dom}{\ensuremath{\mathrm{D}}}

\newcommand{\scalprod}[2]{\ensuremath{\langle #1, #2\rangle}}
\newcommand{\la}{\langle}
\newcommand{\ra}{\rangle}

\newcommand{\floor}[1]{\left\lfloor #1 \right\rfloor}
\newcommand{\ceil}[1]{\left\lceil #1 \right\rceil}

\newcommand{\de}{\ensuremath{\,\mathrm d}}
\renewcommand{\d}{\mathrm d}

\newcommand{\sfe}{\mathsf e} 
\newcommand{\sfv}{\mathsf v} 
\newcommand{\sfx}{\mathsf x}

\newcommand{\N}{\mathbb{N}}
\renewcommand{\P}{\mathbb P} 
\newcommand{\R}{\mathbb{R}} 

\newcommand{\aalpha}{\boldsymbol\alpha}
\newcommand{\bb}{\boldsymbol b} 
\newcommand{\ii}{\boldsymbol i}
\newcommand{\eeta}{\boldsymbol \eta}
\newcommand{\ff}{\boldsymbol f}
\newcommand{\ggamma}{\boldsymbol\gamma}
\newcommand{\mmu}{\boldsymbol\mu}
\newcommand{\rrho}{\boldsymbol\rho}
\renewcommand{\ss}{\boldsymbol s}
\newcommand{\ssigma}{\boldsymbol\sigma} 
\newcommand{\ttheta}{\boldsymbol\vartheta}

\newcommand{\Bb}{\boldsymbol B} 
\newcommand{\fF}{{\boldsymbol F}}
\newcommand{\Gg}{\boldsymbol G}
\newcommand{\Llambda}{\boldsymbol\Lambda}
\newcommand{\Ttheta}{\boldsymbol\Theta} 

\newcommand{\rsqm}[1]{\mathsf m_2(#1)} 
 
\newcommand{\prob}{\ensuremath{\mathcal{P}}} 
\DeclareMathOperator{\supp}{supp}

\newcommand{\U}{\mathsf U} 
\newcommand{\X}{\mathsf X} 
\newcommand{\TX}{\mathsf {T\kern-1.5pt X}}
\newcommand{\UU}{\mathbb U} 

\newcommand{\interval}{\mathcal I} 

\newcommand{\bry}[1]{\boldsymbol b_{#1}}
\newcommand{\bri}[1]{\operatorname{bar}\left (#1\right )}

\newcommand{\bram}[2]{\ensuremath{\left [ #1, #2\right ]_{r}}}

\newcommand{\cB}{{\mathcal B}}
\newcommand{\frF}{{\boldsymbol{\mathrm F}}}

\newcommand{\Sp}[1]{\mathcal{S}\left(#1\right)}
\newcommand{\mmo}{{\boldsymbol B}}
\newcommand{\Sgp}{\boldsymbol S}
\newcommand{\Rest}{\mathsf R}

\newcommand{\finalstep}[2]{{\mathrm N(#1,#2)}}

\newcommand{\EVI}{{\rm EVI}\xspace} 
\newcommand{\IDF}{{\rm IDF}\xspace} 
\newcommand{\MPVF}{{\rm MPVF}\xspace} 
\newcommand{\SDF}{{\rm SDF}\xspace} 
\newcommand{\PVF}{{\rm PVF}\xspace}

\title[Stochastic Euler Schemes and Dissipative Evolutions in $\prob_2$]{Stochastic Euler Schemes and Dissipative Evolutions in the Space of Probability Measures}

\author{Giulia Cavagnari}
\address{Giulia Cavagnari: Politecnico di Milano, Dipartimento di Matematica, Piazza Leonardo Da Vinci 32, 20133 Milano (Italy)}
\email{giulia.cavagnari@polimi.it}

\author{Giuseppe Savar\'e}
\address{Giuseppe Savar\'e: Bocconi University,
  Department of Decision Sciences and BIDSA, Via Roentgen 1, 20136 Milano (Italy)}
\email{giuseppe.savare@unibocconi.it}

\author{Giacomo Enrico Sodini}
\address{Giacomo Enrico Sodini: Institut für Mathematik - Fakultät für Mathematik - Universität Wien, Oskar-Morgenstern-Platz 1, 1090 Wien (Austria)}
\email{giacomo.sodini@univie.ac.at}
\subjclass{Primary: 34A06, 47B44, 49Q22; Secondary: 34A12, 34A60}

 \keywords{Optimal transport, measure differential equations/inclusions in Wasserstein spaces, probability vector fields, dissipative operators, stochastic gradient descent.}

\begin{document}

\begin{abstract} We study the convergence of stochastic time-discretization schemes for evolution equations driven by random velocity fields, including examples like stochastic gradient descent and interacting particle systems. Using a unified framework based on Multivalued Probability Vector Fields, we analyze these dynamics at the level of probability measures in the Wasserstein space. Under suitable dissipativity and boundedness conditions, we prove that the laws of the interpolated trajectories converge to those of a limiting evolution governed by a maximal dissipative extension of the associated barycentric field. This provides a general measure-theoretic study for the convergence of stochastic schemes in continuous time.
\end{abstract}
\maketitle
\tableofcontents
\thispagestyle{empty}

\section{Introduction}
The aim of this paper is to introduce and develop a general framework for studying the convergence of trajectories generated by a suitable stochastic version of the explicit Euler method, which we use 
to approximate the evolution of interacting particle systems governed by dissipative fields in Hilbert spaces.

\subsubsection*{Superposition of autonomous vector fields.}
In order to keep the exposition simple, we first explain the main ideas  of our approach in the case of a simple system of ordinary differential equations in $\R^d$ driven by 
a continuous 
 vector field $b: \R^d \to \R^d$,
\begin{equation}\label{eq:ode}
\dot{x}(t) = b(x(t)), \quad t\in [0,T],  \quad x(0) = \bar{x} \in \R^d   
\end{equation}
satisfying, for some $\lambda\in\R$, a one-sided Lipschitz condition
\begin{equation}
\label{eq:one-sided-Lip}
        \la b(x_0)-b(x_1),x_0-x_1\ra\le \lambda 
    |x_0-x_1|^2\quad\text{for every }x_0,x_1\in \R^d.
\end{equation}
We suppose that $b$ arises as a \emph{stochastic} superposition of
a family of vector fields $g:\R^d\times \U\to\R^d$
depending on
a random parameter $u\in \U$ in  
a probability space $\U$ endowed with 
a probability measure $\mathbb{U}$:
\[ b(x)=
\int_\U g(x,u)\,\d\mathbb{U}(u)=
\mathbb{E}_{\mathbb U} [ g(x,\cdot)]\quad \text{for every } x \in \R^d,\]
with
\begin{displaymath}
    \int_\U |g(x,u)|^2\,\d\mathbb{U}(u)=
\mathbb{E}_{\mathbb U} \Big[ |g(x,\cdot)|^2\Big]
\le L(1+|x|^2).
\end{displaymath}
We can perform a natural time discretization of \eqref{eq:ode} by 
selecting a sequence of uniform step sizes $\tau=T/N$, $N\in \N$, and a sequence 
$(V^{n})_n $ 
 of independent,  $\U$-valued random variables 
 with law $\mathbb U$,
 defined in some probability space
$(\Omega,\P)$. A stochastic approximation $X^n_\tau$ of 
$x(n\tau)$ 
is a family of $\R^d$-valued random variables defined on $\Omega$, which
can then be obtained by the 
following version of the explicit Euler method
\begin{equation}\label{eq:scheme1}
X^{n+1}_\tau = X^n_\tau + \tau g(X^n_\tau, V^n), \quad n = 0,1,\cdots, N;
\quad
X^0_\tau:=\bar x.
\end{equation}
A special case of the above framework occurs when 
\begin{equation}
\label{eq:gradient}
    b =-\nabla h,\quad 
    h(x)=\frac{1}{K} \sum_{u=1}^K H_u(x)
\end{equation}
is the sum of gradients of functions $H_u: \R^d \to \R$. 
In this case, a deterministic approach to approximate \eqref{eq:ode} may be difficult to implement numerically when $K$ is large. For this reason, stochastic approximation techniques have been developed, inspired by the seminal works of Robbins and Monro \cite{RM51}  and of Kiefer and Wolfowitz \cite{KW52}, and they have witnessed a renewed interest, largely motivated by their growing relevance in machine learning contexts (cf. e.g. \cite{RHW86,Bottou10,MB11}). 
In this regard, 
the simplest stochastic formulation takes $\U = \{1, \dots, K\}$, $\mathbb{U}$ to be the uniform distribution on $\U$, and $g(x,u):= -\nabla H_u(x)$; the scheme above in \eqref{eq:scheme1} with $\tau$ fixed 
coincides then with the classical Stochastic Gradient Descent (SGD) method
and one aims to pass to the limit as $n\to\infty$ in order to find a minimum 
(or a stationary point) of $h$.
The effectiveness of this approach relies on the reduction of the computational cost at each iteration.

Instead of
the asymptotic limit as $n\to\infty$,
we want to study the uniform approximation
in $[0,T]$
of the deterministic trajectory of 
\eqref{eq:ode} as $\tau=T/N$ vanishes. 

\subsubsection*{Interacting particle systems}
In our analysis, we can include much more general 
evolution processes, as the ones involving interactions.
In the simplest case,
we may consider a system of $M$ particles $\{x^\omega\}_{\omega=1}^M$ and an 
interaction field $f: \R^d \times \R^d \to \R^d$ determining the velocity of each particle based on the positions of all particles in the system according to 
\begin{equation}
\label{eq:interacting-ode}
\dot{x}^\omega(t) =\frac{1}{M} \sum_{m=1}^M f(x^\omega (t), x^m(t)), \quad x^\omega (0)=\bar x^\omega ,\quad \omega =1,\cdots, M,
\quad t\in [0,T].
\end{equation}
More generally, we can assume that the initial configuration of the particles is described by the law of a random variable $\bar X$:
in this case, we can formulate the evolution of 
the system described by the random variables $X(t,\cdot)$ 
as 
\begin{equation}
    \label{eq:interacting-ode2}
    \dot X(t,\omega)=
    \int_\Omega f\big(X(t,\omega),X(t,\omega')\big)\,\d\mathbb P(\omega'),\quad
    X(0,\omega)=\bar X(\omega)
    \quad\text{for every }\omega\in \Omega.
\end{equation}
We can then consider the approximation scheme 
\begin{equation}\label{eq:scheme2}
X^{n+1}_\tau = X^n_\tau + \tau f(X^n_\tau, Y^{n}_\tau), \quad n=0,1,\cdots,N,
\end{equation}
where $Y^{n}_\tau$ is an independent copy of $X^n_\tau$ and it also independent
of $Y^{k}_\tau$, $k=0,\cdots,n-1$.
\medskip

For both schemes \eqref{eq:scheme1} and \eqref{eq:scheme2}, we can construct a piecewise linear (random) curve $X_\tau$ on the interval $[0,T]$ by linearly interpolating the values  $(X^n_\tau)_{n}$ at the nodes $n\tau$. A natural question arises regarding the convergence of $X_\tau$ as $\tau \downarrow 0$ to the deterministic solution
of 
\eqref{eq:ode}  
(cf. \cite{Benaim99,KY03})
and to that of \eqref{eq:interacting-ode}, \eqref{eq:interacting-ode2}. 

Of course, many variations of the two examples 
above can be considered—for example, allowing the field $g$ to depend on the law of $X^n_\tau$ or assuming that also the field $f$ in the second case arises as a superposition.  In fact, we interpret these evolutions as particular cases of a broader framework, which we use to analyze the behavior of $X_\tau$ as $\tau \downarrow 0$.

The general idea, developed in this paper, is to regard these evolutions as occurring in the Wasserstein space of probability measures $\prob_2(\R^d)$, focusing on the law of $X^n_\tau$ rather than on the random variable itself. We lift the vector fields $b$, $g$, and $f$ to corresponding Probability Vector Fields (\PVF), a notion of distributed velocity field acting on probability measures, introduced in \cite{Piccoli_2019,Piccoli_MDI,Camilli_MDE} and studied in our previous works \cite{CSS, CSS2grande}, also considering the multivalued case. For a comprehensive review of the literature on the study of evolution equations in the Wasserstein space of probability measures, particularly in relation to optimal control problems, we refer the reader to \cite{Piccoli-soa23}.

\subsubsection*{Evolution of probability measures driven by probability vector fields}
To illustrate this approach,  
let us consider the general case of 
the evolution of probability measures 
with bounded support driven by a 
continuous nonlocal vector field
$\boldsymbol b:\X \times \prob_b(\X)\to \X$:
here $\X=\R^d$ for simplicity,
(though we can also handle the case of an infinite-dimensional separable Hilbert space)
and $\prob_b(\X)$ denotes the space
of probability measures with bounded support.

We assume that $\boldsymbol{b}$ is $\lambda$-dissipative \cite{CSS}, i.e.
\begin{equation}
    \label{eq:tot-diss-intro}
    \int \langle \boldsymbol{b}(x_1,\mu_1)
    -\boldsymbol{b}(x_0,\mu_0),
    x_1-x_0\rangle\,\d\mmu(x_0,x_1)
    \le 
    \lambda 
    W_2^2(\mu_0,\mu_1) 
\end{equation}
for every $\mu_0,\mu_1\in \prob_b(\X)$ 
and for some optimal coupling $\mmu\in \Gamma_o(\mu_0,\mu_1)$ for the $L^2$-Wasserstein distance $W_2$,
a condition which is a natural metric-generalization to probability measures of the classical dissipativity (or anti-monotonicity) for operators in Hilbert spaces \cite{BrezisFR}.
By \cite[Theorem 4.2]{CSS2grande} 
for every initial datum $\bar\mu\in \prob_b(\X)$ we can find a unique solution
$\mu\in \rmC([0,T];\prob_b(\X))$
satisfying the continuity equation
\begin{equation}
    \label{eq:cont-eq-intro}
    \partial_t \mu_t +\nabla\cdot\big(\boldsymbol{b}(\cdot,\mu_t)\mu_t\big)=0
    \quad \text{in }(0,T)\times \X,\quad 
    \mu_{t=0}=\bar \mu.
\end{equation}
Selecting an initial random variable
$\bar X$ on a probability space $(\Omega, \P)$ such that $\bar X_\sharp\P=\bar \mu$ (here $\sharp$ is the push-forward operator), 
the measures $\mu_t$, $t\in [0,T]$, can be equivalently characterized as the laws
of the random variables $X(t,\cdot)$ satisfying the 
systems of nonlocal ODEs
\begin{equation}
    \label{eq:nonlocal-ODE-intro}
    \dot X(t,\omega)=
    \boldsymbol{b}\big(X(t,\omega),X(t,\cdot)_\sharp\P\big),\quad 
    X(0,\omega)=\bar X(\omega).
\end{equation}
We can then consider the characteristic map
$\mathrm X:\Omega\to \rmC([0,T];\X)$:
for (almost) every $\omega\in \Omega$, 
$\mathrm X(\omega)$ is the curve 
$t\mapsto X(t,\omega)$ i.e.~a characteristic 
of the PDE \eqref{eq:cont-eq-intro}.
The law $\eeta:=\mathrm X_\sharp\P$ 
is therefore a  probability measure in 
$\prob(\rmC([0,T];\X))$
concentrated on the characteristics of 
\eqref{eq:cont-eq-intro}.

In order to describe the stochastic 
approximation of 
\eqref{eq:nonlocal-ODE-intro}, 
we suppose that $\boldsymbol b$
can be represented as the barycenter
of a Probability Vector Field (PVF) 
$\frF$, which in this particular case 
is a map
from $\prob_b(\X)$ to 
$\prob_b(\X\times \X)
\cong \prob_b(\TX)$ (here $\TX\cong \X\times \X$ denotes the flat tangent bundle to $\X$, which can be interpreted as the natural space for 
the space-velocity pairs $(x,v)$),
satisfying the structural condition
\begin{equation}
    \label{eq:structure-intro}
    \pi^\X_\sharp \frF[\mu]=\mu
    \quad\text{for every $\mu\in \prob_b(\X)$,} 
\end{equation}
where $\pi^\X(x,v)=x$ is the projection on $\X$.
The disintegration $\frF[\mu]
=\int \Phi_x\,\d\mu(x)$
of $\frF[\mu]$ with respect to its first marginal $\mu$ 
yields a Borel family $(\Phi_x)_{x\in \X}$ of probability distributions on 
velocities $\Phi_x\in \prob_b(\X)$ 
at $\mu$-a.e. $x$, whose mean value is precisely $\boldsymbol b(x,\mu):$
\begin{equation}
    \boldsymbol{b}(x,\mu)=
    \int v\,\d\Phi_x(v)\quad\text{for $\mu$-a.e. $x\in \X$}.
\end{equation}

In the simple case \eqref{eq:ode}
we have 
\begin{equation}\label{eq:intrompvf}
 \boldsymbol b(x,\mu)=b(x),\quad
 \frF[\mu] = (\pi^\X, g)_\sharp (\mu \otimes \mathbb{U}), \quad \mu \in \prob_b(\X).   
\end{equation}
In the case of the interacting particle system
\eqref{eq:interacting-ode},
we have
\begin{displaymath}
    \boldsymbol b(x,\mu)=
    \int f(x,y)\,\d\mu(y),\quad 
    \frF[\mu]=(\boldsymbol i_\X,f)_\sharp 
    (\mu\otimes\mu),
\end{displaymath}
being $i_\X$ is the identity on $\X$.
Assuming that 
there exists a constant $a\ge0$ such that 
\begin{equation}
    \label{eq:bound-intro}
    \la v,x\ra\le a(1+|x|^2)
    \quad\text{for $\frF[\mu]$-a.e.~$(x,v)\in \TX$,}
\end{equation}
and for every $R>0$ 
the support of $\frF[\mu]$
is uniformly bounded whenever 
$\supp(\mu)\subset B_R(0),$
the theory developed in \cite{CSS} well applies to the evolution driven by the \PVF $\frF$. The latter can be approximated by  a measure-theoretic version of the Explicit Euler scheme: given 
a step size $\tau=T/N$ 
and an approximation $\bar \mu_\tau \in \prob_b(\X)$ of $\bar \mu$, we iteratively define
\begin{equation}\label{eq:EEFex1}
 M_\tau^0:=\bar\mu_\tau, \quad M_\tau^{n+1} = \exp^\tau_\sharp \frF[M_\tau^n],
 \quad n=0,1,\cdots, N-1,
 \end{equation}
where $\exp^\tau: \TX \to \X$ sends $(x,v)$ to $x + \tau v$. This scheme corresponds to the scheme \eqref{eq:scheme1}
(resp.~\eqref{eq:scheme2}) at the level of laws, indeed for each $n,\tau$ we have that $M_\tau^n$ is the law of the random variable $X^n_\tau$ in \eqref{eq:scheme1} (resp.~\eqref{eq:scheme2}).

In our earlier work \cite{CSS}, we studied the properties of the scheme \eqref{eq:EEFex1} and the uniform convergence of the corresponding piecewise-constant interpolants $t\mapsto M_\tau(t)$ in the space of measures. In particular, we are also able to characterize the limit curve $\mu\in\rmC([0,T];\prob_2(\X))$ in terms of an Evolution Variational Inequality associated with $\frF$.

Comparing this result with the theory in \cite[Theorems 4.2, 4.4]{CSS2grande} (cf.~Theorem \ref{prop1}), this immediately gives that
    $\mu$ 
coincides with the (unique) solution of 
\eqref{eq:cont-eq-intro}
generated by the 
barycenter $\boldsymbol b$ of $\frF$.

It is then natural to lift also the Explicit Euler scheme for $\frF$, in the form of its linear interpolation $M_\tau(\cdot)$, to a probability measure on curves, that we call $\eeta_\tau \in \prob(\rmC([0,T]; \X))$, such that its evaluation at times $n\tau$ matches $M_\tau^n$ (see Definition \ref{def:intEE} for details).
The measure $\eeta_\tau $
can be interpreted as the law
of the piecewise linear interpolation $X_\tau$ of 
a 
sequence $(X^n_\tau)_n$ 
of random variables 
satisfying the following properties:
\begin{itemize}
    \item 
    the law of $X^n_\tau$ is $M^n_\tau$:
    $(X^n_\tau)_\sharp \P=M^n_\tau$, $n=0,\cdots,N$,
    and $X^0_\tau\to \bar X$ in $L^2$
    as $\tau\downarrow0$,
    \item 
    $(X^n_\tau)_n$ is 
    a Markov chain,
    \item 
    denoting by $V^n_\tau$ the ``forward'' discrete velocity
    $\tau^{-1}(X^{n+1}_\tau-
X^n_\tau)$, the joint law 
of $\big(X^n_\tau,V^n_\tau\big)$ 
is precisely $\frF[M^n_\tau]$.
\end{itemize}
We are therefore led to study the convergence of $\eeta_\tau$, which corresponds to the convergence of the process $X_\tau$. 
Our main result is that 
$\eeta_\tau$ converges to $\eeta$ in the strong $L^2$-Wasserstein topology on $\prob_2(\rmC([0,T]; \X))$ (notice that, in this case, the convergence would not just be limited to the time evaluations, as already shown).

While compactness of the family $(\eeta_\tau)_\tau$ in the weak topology of probability measures 
follows from \cite{CSS} with the dissipativity and boundedness conditions, characterizing the limit is one main contribution of this paper; see in particular Theorem \ref{thm:main}. 
 In this result we show that the convergence takes place actually in the strong $L^2$-Wasserstein topology on $\prob_2(\rmC([0,T]; \X))$ and that the limit $\eeta$ is concentrated precisely on the solutions 
 to \eqref{eq:nonlocal-ODE-intro}.
 The tools we employ to characterize $\eeta$ are largely based on our second work \cite{CSS2grande} where we study, among other things, the properties of (multivalued) {\PVF}s enjoying a \emph{total dissipativity} condition.
 
 As a byproduct, we obtain the convergence 
 in $L^2(\Omega, \P; \rmC([0,T]; \X))$ of $X_\tau$
 to the unique solution of the deterministic ODE in \eqref{eq:nonlocal-ODE-intro} driven by $\boldsymbol b$
(Theorem \ref{prop:young}).
We will further extend the result described above in a much greater generality so to treat simultaneously also the case of multivalued and possibly discontinuous 
probability vector fields.
These results fully clarify the links between the explicit Stochastic Euler Method \eqref{eq:EEFex1}
induced by a metrically dissipative (multivalued) \PVF $\frF$ 
and the contraction semigroup generated by its totally dissipative barycentric 
projection, which in turn can be studied by
the implicit Euler Scheme. 

\quad \\
\textbf{Plan of the paper.} The paper is organized as follows. Section \ref{sec:Ecomp} contains the preliminaries on Wasserstein spaces, the main notations adopted, a recall of the definition of multivalued \PVF (\MPVF) and notions of dissipativity taken from \cite{CSS,CSS2grande}. In Section \ref{sec:EE}, we recall the Explicit Euler scheme for dissipative \MPVF{s} and construct a probabilistic representation of its affine interpolants. We then study its convergence as the step-size vanish. In Section \ref{sec:IE}, we study the stability of the probabilistic representation of the Implicit Euler scheme for a maximal totally dissipative \MPVF with respect to the initial data. Finally, our main results are contained in Section \ref{sec:final}, while illustrative examples appear in Section \ref{sec:examples}. Appendices \ref{sec:appA} and \ref{sec:sticky} close the work.

\mysubsubsection{Acknowledgments.} G.C. acknowledges the partial support of INDAM-GNAMPA project 2024 \emph{Controllo ottimo infinito dimensionale:
aspetti teorici ed applicazioni} (CUP\_E53C23001670001).

\section{Preliminaries}\label{sec:Ecomp}

We introduce here the main notations used throughout the paper.

Given $\lambda\in\R$, we write $\lambda_+:=\max\{0,\lambda\}$ and denote by $\floor{\cdot}$ the floor function. Given a set $X$, we denote by $\ii_X:X\to X$ the identity map on $X$.
Given $M\in\N$ and sets $X_0,\dots,X_M$, we define the \emph{$i$-th projection map} by
\[\pi^i:X_0\times X_1\times\dots\times X_M\to X_i,\quad \pi^i(x_0,\dots,x_M)=x_i,\]
with $i=0,\dots,M$; note that we start counting from the index $0$, this is done to be coherent with the time evaluations of curves defined in some interval of the form $[0,T]$. Similarly, we define
\[\pi^{i,j}:X_0\times X_1\times\dots\times X_M\to X_i\times X_j,\quad \pi^{i,j}(x_0,\dots,x_M)=(x_i,x_j),\]
for $i,j=0,\dots,M$.
\medskip

To ensure generality, we introduce the basic notions of measure theory considering $X$ and $Y$ to be Lusin completely regular topological spaces. Recall that a topological space $X$ is \emph{completely regular} if it is Hausdorff and, for every closed set $C$ and point $x \in X \setminus C$, there exists a continuous function $f: X \to [0,1]$ such that $f(C) = \{1\}$ and $f(x) = 0$. This framework is well-suited to our analysis of Borel probability measures on (subsets of) a separable Hilbert space $X$, which may carry either the strong or weak topology.

With $\prob(X)$ we denote the set of Borel probability measures on $X$ endowed with the weak/narrow topology, induced by the duality with $\rmC_b(X)$, the space of bounded continuous real-valued functions defined on $X$.

Given $\mu\in\prob(X)$ and a Borel map $r:X\to Y$, we define the \emph{push-forward} measure $r_\sharp\mu\in\prob(Y)$ by
\[\int_Y \varphi(y)\de(r_\sharp\mu)(y)=\int_X \varphi(r(x))\de\mu(x),\]
for any $\varphi:Y\to\R$ bounded Borel map.

For later use, we recall the \emph{disintegration theorem} (c.f. \cite[Theorem 5.3.1]{ags}).
\begin{theorem}\label{thm:disintegr}
Let $\mathbb X, X$ be Lusin completely regular topological spaces, $\mmu\in\prob(\mathbb{X})$, $r:\mathbb{X}\to X$ be a Borel map. Then, there exists a $(r_{\sharp}\mmu)$-a.e.~uniquely determined Borel family of probability measures $\{\mmu_x\}_{x\in X}\subset\prob(\mathbb{X})$ such that $\mmu_x(\mathbb{X}\setminus r^{-1}(x))=0$ for $(r_{\sharp}\mmu)$-a.e. $x\in X$, and
\[\int_{\mathbb{X}}\varphi(\boldsymbol{x})\de\mmu(\boldsymbol{x})=\int_X\left(\int_{r^{-1}(x)}\varphi(\boldsymbol{x})\de\mmu_x(\boldsymbol{x})\right)\de(r_{\sharp}\mmu)(x)\]
for every bounded Borel map $\varphi:\mathbb{X}\to\R$.
\end{theorem}
\begin{remark}\label{rmk:disintegr}
When $\mathbb{X}=X_0\times X_1$ and $r$ is the projection $\pi^0$ on the first component, we can
identify the disintegration $\{\mmu_x\}_{x\in X_0} \subset
\prob(\mathbb{X})$ of $\mmu\in\prob(X_0\times X_1)$ w.r.t.~the map $\pi^0$
with a family of probability measures $ \{\mu_{x_0}\}_{x_0\in X_0} \subset
\prob(X_1)$. We write
$\mmu= \displaystyle\int_{X_0}\mu_{x_0}\,\d ({\pi^0}_{\sharp}\mmu)(x_0)$.
\end{remark}

The set of admissible couplings (transport plans) between two probability measures $\mu\in\prob(X)$, $\nu\in\prob(Y)$ is given by
\[\Gamma(\mu,\nu):= \left  \{ \ggamma \in \prob(X \times Y) \mid {\pi^{0}}_{\sharp} \ggamma = \mu \, , \, {\pi^{1}}_{\sharp} \ggamma = \nu \right \}.\]
If $\ggamma\in\prob(X\times Y)$, we say that the measures ${\pi^0}_\sharp\ggamma$ and ${\pi^1}_\sharp\ggamma$, are the first and second marginals of $\gamma$, respectively.

If $(X,|\cdot|)$ is a separable normed space, given $\mu\in\prob(X)$ we define its \emph{$2$-moment} by
\[\rsqm \mu :=\left( \int_\X |x|^2 \de \mu(x)\right)^{1/2},\]
and denote by $\prob_2(X)$ the set of all measures in $\prob(X)$ with finite $2$-moment.
\medskip

If $(X,|\cdot|)$ is a complete and separable normed space,
we denote by $W_2$ the Wasserstein distance over $\prob_2(X)$, defined by
\begin{align} \label{eq:w21} W_2^2(\mu, \mu') &:= \inf \left \{ \int_{X \times X} |x-y|^2 \de \ggamma(x,y) \mid \ggamma \in \Gamma(\mu, \mu') \right \},
\end{align}
and refer to \cite[\S 7]{ags} for the main properties of the complete and separable metric space $(\prob_2(X),W_2)$.

We denote by $\Gamma_o(\mu, \mu')$ the subset of couplings in $\Gamma(\mu, \mu')$ realizing the infimum in \eqref{eq:w21}.
\medskip

In the sequel, $\X$ denotes a separable Hilbert space with
norm $|\cdot|$ and scalar product $\la\cdot,\cdot\ra$.
In case $X=\rmC(\interval;\X)$ is endowed with the uniform norm $\|\cdot\|_\infty$ and $\interval\subset\R$ is a compact interval, we use the notation $W_{2,\infty}$ to denote the Wasserstein distance over $\prob_2(X)$.
\medskip

We denote by $\TX$ the tangent bundle to $\X$, identified with the cartesian product $\X\times\X$
with the induced norm $|(x,v)|:=\big(|x|^2+|v|^2\big)^{1/2}$ and the 
strong-weak topology, i.e.~the product of the strong topology on the first component and the weak topology on the second one. The set $\prob(\TX)$ is defined thanks to the identification of
$\TX$ with $\X\times \X$ and it is endowed with
the narrow topology induced by the strong-weak topology in $\TX$. Given $\Phi \in \prob(\TX)$,   we define the partial $2$-moment of $\Phi$ as
\[|\Phi|_2:=\left( \int_{\TX} |v|^2 \de \Phi(x,v)\right)^{1/2}.\]

Following the characterization result provided in \cite[Section 2.1]{CSS2grande}, we give the following definition of convergence in $\prob_2^{sw}(\TX)$.
\begin{definition}\label{def:sw}
    Given $(\Phi_n)_{n \in\mathbb N}\subset \prob_2(\TX)$, $\Phi \in \prob_2(\TX)$, we say that $\Phi_n\to\Phi$ in $\prob_2^{sw}(\TX)$ as $n\to\infty$ if all the following are satisfied
      \begin{enumerate}
      \item $\Phi_n\to\Phi$ in $\prob(\TX)= \prob(\X^s\times \X^w)$, where $\X^s$ is endowed with the strong topology, while $\X^w$ with the weak,
      \item $\displaystyle \lim_{n \to + \infty}\int |x|^2\,\d\Phi_n(x,v)=\int |x|^2\,\d\Phi(x,v)$,
      \item $\displaystyle \sup_{n} |\Phi_n|_2<\infty$.
\end{enumerate}
\end{definition}

\medskip

In $\TX$ we will use the following notation for projection maps: we set
\begin{align*}
\sfx:\TX\to\X,\quad\sfx(x,v):=x;&\qquad\sfv:\TX\to\X,\quad\sfv(x,v)=v;\\
\sfx^i:\TX\times\TX\to\X,\quad\sfx^i(x_0,v_0,x_1,v_1):=x_i;&\qquad\sfv^i:\TX\times\TX\to\X,\quad\sfv^i(x_0,v_0,x_1,v_1)=v_i;
\end{align*}
for $i=0,1$. Given $\interval\subset\R$ an interval of $\R$ and $t\in\interval$, we define the \emph{exponential map} $\exp^t:\TX\to\X$,
\[\exp^t(x,v)=x+tv,\]
and the \emph{evaluation map} $\sfe_t:\rmC(\interval;\X)\to\X$ as
\[\sfe_t(\gamma)=\gamma(t).\]

\subsection{Dissipative Probability Vector Fields and \EVI solutions}
From now on, $\X$ denotes a separable Hilbert space.
We recall here the definition of \MPVF and different notions of dissipativity: the metric dissipativity introduced in \cite{CSS}, where we provide well-posedness results of the associated evolution equation in the Wasserstein space, using a measure-theoretic Explicit Euler scheme; the total dissipativity, employed in \cite{CSS2grande} to deal with an Implicit Euler scheme.
\begin{definition}[Multivalued Probability Vector Field - \MPVF] \label{def:MPVF}
  A \emph{multivalued probability vector field} $\frF$ is a nonempty subset of
  $\prob_2(\TX)$ with domain $\dom(\frF) := \sfx_\sharp(\frF)=
  \{ \sfx_\sharp\Phi:\Phi\in \frF \}$.
  Given $\mu \in \prob_2(\X)$, we define the \emph{section} $\frF[\mu]$
  of $\frF$ as 
  \[ \frF[\mu] := \left \{ \Phi \in \frF
      \mid \sfx_{\sharp}\Phi = \mu \right \}. \]
\end{definition}

Given $\Phi\in\prob_2(\TX)$, the \emph{barycenter of $\Phi$} is the function $\bry\Phi\in L^2(\X,\sfx_\sharp\Phi;\X)$ defined by 
\begin{equation}\label{eq:bry}
\bry\Phi(x):=\int v\,\d\Phi(x,v)=\int v\,\d\Phi_x(v),
\end{equation}
where $\{\Phi_x\}_{x\in\X}\subset\prob_2(\X)$ is the disintegration of $\Phi$ w.r.t. $\sfx_\sharp\Phi$.

\begin{definition}\label{def:baryc}
Let $\frF\subset\prob_2(\TX)$ be a \MPVF. We define the \MPVF $\bri\frF$ by
\begin{equation}\label{eq:brimpvf}
\bri\frF[\mu]:= \left \{ \bri\Phi:=(\ii_\X, \bry\Phi)_\sharp\mu : \Phi \in \frF[\mu] \right \}, \quad \mu \in \dom(\frF).
\end{equation} 
\end{definition}

\medskip

Before we recall the metric and total notions of dissipativity for a \MPVF $\frF$,
we briefly recall the definition of duality pairings between measures in $\prob_2(\TX)$.

\begin{definition}[Metric-duality pairings] \label{def:pairings} For every $\Phi_0, \Phi_1 \in \prob_2(\TX)$, $\mu_1 \in \prob_2(\X)$, we set
\begin{align*}
    \Lambda(\Phi_0, \mu_1)&:= \left \{ \ssigma \in \Gamma(\Phi_0, \mu_1)
    \mid
    (\sfx^0,\sfx^1)_{\sharp}\ssigma \in  \Gamma_o(\sfx_\sharp \Phi_0, \mu_1) \right
      \},\\
      \Lambda(\Phi_0, \Phi_1)&:= \left \{ \Ttheta \in \Gamma(\Phi_0, \Phi_1)
                         \mid (\sfx^0,\sfx^1)_{\sharp} \Ttheta\in
                          \Gamma_o(\sfx_\sharp \Phi_0, \sfx_\sharp \Phi_1)
    \right \},
  \end{align*}
  where, with a slight abuse of notation, we used $\sfx^1$ to be also the projection map $\sfx^1(x_0,v_0,x_1)=x_1$ acting on $\TX\times\X$.
  We set
  \begin{align*}
    \bram{\Phi_0}{\mu_1} &:= \min \left \{ \int_{\TX \times
                                   \X} \scalprod{x_0-x_1}{v_0} \de \ssigma
                                   \mid \ssigma \in \Lambda(\Phi_0,
                                   \mu_1) \right \}, \\
\bram{\Phi_0}{\Phi_1} &:= \min \left \{ \int_{\TX \times \TX} \scalprod{x_0-x_1}{v_0-v_1} \de \Ttheta \mid \Ttheta \in \Lambda(\Phi_0, \Phi_1) \right \}.
  \end{align*}
\end{definition}

\begin{definition}[Metric $\lambda$-dissipativity] \label{def:dissipative}
  A \MPVF $\frF \subset \prob_2(\TX)$ is (metrically) \emph{$\lambda$-dissipative},
  $\lambda\in \R$, if
    \begin{equation}
      \bram{\Phi_0}{\Phi_1} \le \lambda W_2^2(\sfx_\sharp\Phi_0,\sfx_\sharp \Phi_1)
      \quad \text{ for every }\,\Phi_0,\Phi_1\in \frF.
  \label{eq:33}
\end{equation}
In case $\lambda=0$, we say that $\frF$ is dissipative.
\end{definition}

\begin{definition}[Total $\lambda$-dissipativity]
    \label{def:total-dissipativity}
    We say that a \MPVF $\frF \subset \prob_2(\TX)$ is \emph{totally $\lambda$-dissipative}, $\lambda\in\R$, if 
    for every $\Phi_0,\Phi_1\in \frF$
    and every $\boldsymbol \vartheta\in \Gamma(\Phi_0,\Phi_1)$ we have
    \begin{equation}
        \label{eq:total-dissipativity}
        \int \langle v_1-v_0,x_1-x_0
        \rangle \,\d\boldsymbol\vartheta(x_0,v_0,x_1,v_1)\le \lambda \int |x_1-x_0|^2\d\boldsymbol\vartheta(x_0,v_0,x_1,v_1).
    \end{equation}
    We say that $\frF$ is \emph{maximal totally $\lambda$-dissipative} if it is 
    maximal in the class of totally $\lambda$-dissipative \MPVF{s}: if $\frF'\supset \frF$
    and $\frF'$ is totally $\lambda$-dissipative, then $\frF'=\frF.$
\end{definition}

We recall the definition of $\lambda$-\EVI solution for a \MPVF. Here, $\interval$ denotes an arbitrary (bounded or unbounded) interval of $\R$.
\begin{definition}[$\lambda$-Evolution Variational Inequality] \label{def:evi}
  Let $\frF$ be a \MPVF and let $\lambda \in \R$.
  We say that a continuous curve $\mu: \interval  \to \overline{\dom(\frF)}$
  is a \emph{$\lambda$-\EVI solution} (in $\interval$) for the \MPVF $\frF$ if 
\begin{equation}\label{eq:EVIdef}
    \begin{aligned}
      \frac12\frac\d{\d t} W_2^2(\mu_t,\sfx_\sharp \Phi)
      &\le \lambda W_2^2(\mu_t,\sfx_\sharp \Phi)-
      \bram{\Phi}{\mu_t} \text{ in } \mathscr{D}'(\intt{\interval}) \text{ for every } \Phi \in \frF,
    \end{aligned}
  \end{equation}
 where when we write $\mathscr{D}'(\intt{\interval})$ it means that the expression holds in the distributional sense in $\intt{\interval}$. Equivalently, \eqref{eq:EVIdef} can be rephrased as 
 \begin{equation*}
    \begin{aligned}
      \frac12\frac\d{\d t}^+ W_2^2(\mu_t,\sfx_\sharp \Phi)
      &\le \lambda W_2^2(\mu_t,\sfx_\sharp \Phi)-
      \bram{\Phi}{\mu_t} \text{ for every $t\in\intt{\interval}$ } \text{ for every } \Phi \in \frF,
    \end{aligned}
  \end{equation*}
  where $\frac\d{\d t}^+$ denotes the right upper Dini derivative.
\end{definition}

\section{Explicit Euler scheme and lifting to curves}\label{sec:EE}

We recall here below the definition of Explicit Euler scheme for a \MPVF $\frF \subset \prob_2(\TX)$ coming from \cite{CSS} and provide a construction of a probabilistic representation of it with a probability measure on continuous paths valued in $\X$.

\begin{definition}[Explicit Euler Scheme]\label{def:EEscheme}
  Let $\frF\subset\prob_2(\TX)$ be a \MPVF and suppose we are given
  a step size $\tau>0$, an initial datum $\bar \mu\in\dom(\frF)$,
  a bounded interval $[0,T]$, corresponding to
  the final step $\finalstep T\tau:=\ceil{T/\tau},$
  and a stability bound $L>0$.
A sequence
  $(M^n_\tau,\Phi_\tau^n)_{0\le n\le \finalstep T\tau}\subset \dom(\frF)\times \frF $ is a
  \emph{$L$-stable solution to the Explicit Euler Scheme in $[0,T]$ starting from $\bar \mu$} if 
\begin{equation}
  \label{eq:EE}
  \tag{EE}
  \left\{
\begin{aligned}
  M_{\tau}^0& = \bar\mu ,\\
  \Phi_\tau^n &\in \frF[M_{\tau}^n],\
  |\Phi_\tau^n|_2 \le L
  &&0\le n<\finalstep T\tau,\\
  M_{\tau}^{n} &= (\exp^{\tau})_{\sharp} \Phi_\tau^{n-1} 
  &&1\le n\le \finalstep T\tau.
\end{aligned}
\right.
\end{equation}
\end{definition}

\begin{definition}[Single-step and multi-step plans]\label{def:step-plans}
Let $\tau>0$.
\begin{itemize}
\item Given $\Phi\in\prob(\TX)$, we define the \emph{single-step plan associated with $\Phi$} as
\[T_\tau\equiv T_\tau(\Phi):=(\sfx,\exp^\tau)_\sharp\Phi\in\prob(\X^2),\]
so that $(\pi^0)_\sharp {T_\tau}=\sfx_\sharp\Phi$. By disintegration theorem, there exists a $\left[\sfx_\sharp\Phi\right]$-a.e.~uniquely determined Borel family of probability measures $\{T_{\tau,x}\}_{x\in\X}\subset\prob(\X)$ such that
    \[ T_\tau = \int_{\X} ( \delta_{x} \otimes T_{\tau,x}) \de \left[\sfx_\sharp\Phi\right](x).\]
    \item Given $N\in\N$ and a family $(\Phi^n)_{0\le n\le N}\subset\prob(\TX)$, we denote the single-step plans associated with each $\Phi^n$ by $T_\tau^n:=T_\tau(\Phi^n)$, $n\in\{0,\dots,N-1\}$, and define by recursion
\begin{align*}
\aalpha_\tau^1&:=T_\tau^0,\\
\aalpha_\tau^n&:=\int ( \delta_{(x_0,x_1,\dots,x_{n-1})} \otimes T_{\tau,x_{n-1}}^{n-1}) \de \aalpha_\tau^{n-1}(x_0,x_1,\dots,x_{n-1})\in\prob(\X^{n+1}),
\end{align*}
for $2\le n\le N$. The \emph{multi-step plan associated with $(\Phi^n)_{0\le n\le N}$} is defined by
\[\aalpha_\tau\equiv\aalpha_\tau\left((\Phi^n)_n\right):=\aalpha_\tau^N\in\prob(\X^{N+1}).\]
\end{itemize}
\end{definition}

\begin{lemma}\label{lem:Ralpha}
Given $N,n\in\N$, $n\le N$, we define the restriction map $\Rest_n:\X^{N+1}\to\X^{n+1}$ as \[\Rest_n(x_0,\dots,x_N)=(x_0,\dots,x_n).\]
Let $\tau>0$ and $(\Phi^n)_{0\le n\le N}\subset\prob(\TX)$. With the notations of Definition \ref{def:step-plans}, we have
\begin{enumerate}
\item\label{item:Ralpha} $(\Rest_n)_\sharp\aalpha_\tau=\aalpha_\tau^n$;
\item\label{item:evalalpha} $(\pi^n)_\sharp\aalpha_\tau^n=\sfx_\sharp\Phi^n$,
\end{enumerate}
for any $n\in\{1,\dots,N\}$, where we recall that $\pi^n:\X^{n+1}\to\X$ is the projection map defined as $\pi^n(x_0,\dots,x_n)=x_n$.
\end{lemma}
\begin{proof}
Item \eqref{item:Ralpha} is immediate by definition. To prove item \eqref{item:evalalpha}, we notice that
\[(\pi^1)_\sharp T_\tau^n=\sfx_\sharp\Phi^{n+1},\quad\text{for any }n\in\{0,\dots,N-1\},\]
and that $(\pi^1)_\sharp\aalpha_\tau^1=(\pi^1)_\sharp T_\tau^0=\sfx_\sharp\Phi^1$.
By induction, assume that item \eqref{item:evalalpha} holds for $n\in\{1,\dots,N-1\}$ and let us prove it for $n+1$. For any bounded Borel function $f:\X\to\R$, we have
\begin{align*}
\int f(x)\de \left((\pi^{n+1})_\sharp\aalpha_\tau^{n+1}\right)(x)&=\int f(x_{n+1})\de \aalpha_\tau^{n+1}(x_0,\dots,x_{n+1})\\
&=\int f(x_{n+1})\de T_{\tau,x_n}^n(x_{n+1}) \de\aalpha_\tau^n(x_0,\dots,x_n)\\
&=\int f(x_{n+1})\de T_{\tau,x_n}^n(x_{n+1}) \de \left((\pi^n)_\sharp \aalpha_\tau^n\right)(x_n)\\
&=\int f(x_{n+1})\de T_{\tau,x_n}^n(x_{n+1}) \de \left[\sfx_\sharp\Phi^n\right](x_n)\\
&=\int f(x_{n+1})\de T_\tau^n(x_n,x_{n+1})\\
&=\int f(x_{n+1})\de \left[\sfx_\sharp\Phi^{n+1}\right](x_{n+1}),\\
\end{align*}
thus proving item \eqref{item:evalalpha}.
\end{proof}

\begin{definition}[Interpolations of \eqref{eq:EE}]\label{def:intEE}
Let $\frF \subset \prob_2(\TX)$ be a \MPVF, $\bar\mu \in \dom(\frF)$ and $\tau, T, L>0$. Let $(M^n_\tau,\Phi_\tau^n)_{0\le n\le \finalstep T\tau}\subset \dom(\frF)\times \frF$ be as in Definition \ref{def:EEscheme} and $\aalpha_\tau\in\prob(\X^{\finalstep T\tau+1})$ be the multi-step plan associated with $\left(\Phi_\tau^n\right)_n$ as in Definition \ref{def:step-plans}. We define the following interpolations of the sequence $(M^n_\tau,\Phi_\tau^n)_n$:
\begin{align}
\label{eq:affineM}
  M_{\tau}(t) &:= (\exp^{t-n\tau})_{\sharp} \Phi_\tau^n\in\prob_2(\X) \text{\quad if } t \in [n\tau, (n+1)\tau]
  \text{ for some } n \in \N,\ 0\le n<\finalstep T\tau,\\
\fF_{\tau}(t) &:= \Phi_\tau^{\floor{t/\tau}}\in\prob_2(\TX), \quad t\in [0,T].\label{eq:piecewiseVel}   
\end{align}

Denote by $G:\X^{\finalstep T\tau+1}\to\rmC([0,T]; \X)$ the map sending points $(x_0, x_1, \dots, x_\finalstep T\tau)$ to their affine interpolation, i.e. to the curve $\gamma$ given by
    \[ \gamma(t):= \frac{1}{\tau}\left((n+1)\tau-t\right) x_n+\frac{1}{\tau}(t-n\tau) x_{n+1},\]
if $t \in [n\tau, (n+1)\tau]\cap[0,T]$, $0\le n<\finalstep T\tau$. We define the \emph{selected probabilistic representation of the affine interpolation of $(M^n_\tau,\Phi_\tau^n)_n$} as
\begin{equation}\label{eq:eeta_tauEE}
\eeta_\tau:= G_\sharp \aalpha_\tau \in \prob(\rmC([0,T]; \X)).
\end{equation}

We define the following (possibly empty) sets
\begin{equation}\label{eq:setsEM}
\begin{split}
\mathscr M(\bar\mu,\tau,T,L)&:=\Big\{M_\tau\mid M_\tau \text{ is the curve given by }\eqref{eq:affineM}\Big\},\\
\mathscr R(\bar\mu,\tau,T,L)&:=\Big\{\eeta_\tau\mid \eeta_\tau \text{ is the measure given by }\eqref{eq:eeta_tauEE}\Big\},\\
\mathscr E(\bar\mu,\tau,T,L)&:=\Big\{(M_\tau,\fF_\tau)\mid M_\tau, \fF_\tau \text{ are the curves given by }\eqref{eq:affineM},\eqref{eq:piecewiseVel}\text{ respectively}\Big\},\\
\mathscr T(\bar\mu,\tau,T,L)&:=\Big\{(M_\tau,\fF_\tau,\eeta_\tau)\mid M_\tau, \fF_\tau,\eeta_\tau \text{ are given by }\eqref{eq:affineM},\eqref{eq:piecewiseVel},\eqref{eq:eeta_tauEE}\text{ respectively}\Big\}.
\end{split}
\end{equation}

\end{definition}

In the following, a frequently requested assumption is the \emph{solvability} of the Explicit Euler scheme at a given $\bar\mu\in \dom(\frF)$. Meaning that, for some $\bar\mu \in \dom(\frF)$, $T,L>0$ and a vanishing sequence of step sizes $\tau(h) \downarrow 0$, we require the sets $\mathscr E(\bar\mu,\tau(h),T,L)$ to be not empty for every $h \in \N$. Sufficient conditions ensuring this property are given e.g. in \cite[Proposition 5.20]{CSS}. 

We recall that the solvability of the Explicit Euler Scheme \eqref{eq:EE}, together with the $\lambda$-dissipativity of $\frF$, provides existence and uniqueness of $\lambda$-\EVI solutions (cf. \cite[Theorem 5.9(3)-(4)]{CSS}).

\begin{theorem} \label{thm:apriori-estimate} Let $\frF$ be a
  $\lambda$-dissipative \MPVF according to \eqref{eq:33}. If $T,L>0$, $\N\ni h\mapsto \tau(h)\in\R_+$ is a vanishing sequence of
  time steps, $\bar\mu\in \dom(\frF)$ and
  $M_h\in \mathscr M(\bar\mu,\tau(h),T,L)$,
  then $M_h$ is uniformly converging to a Lipschitz continuous limit curve
  $
  \mu:[0,T]\to\overline{\dom(\frF)}$ which
  is the unique $\lambda$-\EVI solution in $[0,T]$ for $\frF$
  starting from $\bar\mu$.
\end{theorem}

 The link between $\mathscr E(\bar\mu,\tau,T,L)$ and $\mathscr R(\bar\mu,\tau,T,L)$, which justifies the definition of the set of triples $\mathscr T(\bar\mu,\tau,T,L)$, is made explicit by the following result.

\begin{proposition}\label{prop:exlift}
    Let $\frF \subset \prob_2(\TX)$ be a \MPVF and let $\bar\mu \in \dom(\frF)$ and $\tau, T, L>0$.  Let $(M_\tau, \frF_\tau,\eeta_\tau)\in \mathscr T(\bar\mu,\tau,T,L)$, then
    \begin{enumerate}
        \item\label{item1-etatau} $\eeta_\tau$ is concentrated on curves starting from $\supp(\bar\mu)$ which are affine in every interval $[n\tau, (n+1)\tau] \cap [0,T]$, $n=0, \dots, \finalstep T\tau$;
        \item\label{item2-etatau} we have\[\left(\sfe_{n\tau},\frac{\sfe_{(n+1)\tau}-\sfe_{n\tau}}{\tau}\right)_\sharp\eeta_\tau=\frF_\tau(\tau n) \quad \text{ for every } n=0, \dots, \finalstep T\tau-2.\]
    \end{enumerate}
    In particular, $(\mathsf{e}_t)_\sharp \eeta_\tau = M_\tau(t)$ for every $t \in [0,T]$.
\end{proposition}
\begin{proof} It is clear by construction that $\eeta_\tau$ satisfies \eqref{item1-etatau}. Let $n \in \{0, \dots, \finalstep T\tau-2\}$ and observe that, for every bounded Borel function $f:\TX\to\R$, we have
\begin{align*}
&\int_{\TX} f(x,v) \de \left[\left(\sfe_{n\tau},\frac{\sfe_{(n+1)\tau}-\sfe_{n\tau}}{\tau}\right)_\sharp\eeta_\tau\right](x,v)\\
&= \int_{\X^{\finalstep T\tau+1}} f \left ( x_n, \frac{x_{n+1}-x_n}{\tau} \right ) \de \aalpha_\tau(x_0,\dots, x_{\finalstep T\tau}) \\
&= \int_{\X^{n+2}} f \left ( x_n, \frac{x_{n+1}-x_n}{\tau} \right ) \de \aalpha_\tau^{n+1}(x_0,\dots, x_{n+1}) \\
&= \int_{\X^{n+2}} f \left ( x_n, \frac{x_{n+1}-x_n}{\tau} \right ) \de T_{\tau, x_n}^n(x_{n+1})\de\aalpha_\tau^n(x_0,\dots, x_n) \\
&= \int_{\X^2} f \left ( x_n, \frac{x_{n+1}-x_n}{\tau} \right ) \de T_{\tau, x_n}^n(x_{n+1})\de M_\tau^n(x_n) \\
&= \int_{\X^2} f \left ( x_n, \frac{x_{n+1}-x_n}{\tau} \right ) \de T_\tau^n(x_n,x_{n+1}) \\
&= \int_{\TX} f(x,v)\de \Phi_\tau^n(x,v)\equiv \int_{\TX} f(x,v)\de \left[\frF_\tau(\tau n) \right](x,v),
\end{align*}
where we used the definitions of $\eeta_\tau$, $\aalpha_\tau^n$, $T_\tau^n$ in Definitions \ref{def:intEE},\ref{def:step-plans} and Lemma \ref{lem:Ralpha}.

This proves \eqref{item2-etatau}.
\end{proof}

\medskip

In the following, we give a counterpart of Theorem \ref{thm:apriori-estimate} for probabilistic representations of the Explicit Euler Scheme.

\begin{proposition}\label{prop:tight} Let $\frF$ be a
  $\lambda$-dissipative \MPVF according to \eqref{eq:33}. Let $T,L>0$, $\N\ni h\mapsto \tau(h)\in\R_+$ be a vanishing sequence of
  time steps, $\bar\mu\in \dom(\frF)$. Let $(M_h, \frF_h) \in \mathscr E(\bar{\mu},\tau(h),T,L)$ and let $\eeta_h \in\prob(\rmC([0,T]; \X))$ satisfying \eqref{item1-etatau},\eqref{item2-etatau} of Proposition \ref{prop:exlift} for $(M_h, \frF_h)_h$. Then the following hold:
\begin{enumerate}
\item the family $(\eeta_h)_h$ is tight in $\prob(\rmC([0,T]; \X))$; in particular it weakly converges, up to subsequences, to a measure $\eeta\in\prob_2(\rmC([0,T]; \X))$;
\item $\mu_t:=(\sfe_t)_\sharp \eeta$, $t \in [0,T]$, is the unique $\lambda$-\EVI solution in $[0,T]$ for $\frF$ starting from $\bar\mu$.
\end{enumerate}
Moreover, if $\supp(\bar\mu)$ is bounded, then $(\eeta_h)_h$ is relatively compact in $\prob_2(\rmC([0,T]; \X))$, so that it converges to $\eeta$ also in $W_{2, \infty}$, up to subsequences.
\end{proposition}
\begin{proof} Let $\mathcal{A}_2: \rmC([0,T]; \X) \to [0,+\infty]$ be the $2$-action functional defined in \eqref{eq:A2}.
By Proposition \ref{prop:exlift}, $M_h(t) = (\sfe_t)_\sharp \eeta_h$, $t \in [0,T]$, and $M_h$ is uniformly converging to $(\tilde\mu_t)_{t \in [0,T]}\subset\mathrm{Lip}([0,T];\X)$ by Theorem \ref{thm:apriori-estimate}. Thus, in order to prove tightness of $(\eeta_h)_h$ it is enough (see e.g.~\cite[Theorem 10.4]{abs21}) to prove that 
\begin{equation}\label{eq:tight}
\sup_h\int \mathcal{A}_2 \de \eeta_h < +\infty.
\end{equation}
By Proposition \ref{prop:exlift}\eqref{item1-etatau},\eqref{item2-etatau}, we have
\begin{align*}
    \int \mathcal{A}_2 \de \eeta_h =  \sum_{k=0}^{\finalstep T{\tau(h)} -1} \int_{\TX} \tau(h) |v|^2 \de\left[\frF_{\tau(h)}(k \tau(h))\right](x,v) \le \tau(h) L^2 \finalstep T{\tau(h)} \le L^2(T + \tau(h)),
\end{align*}
which yields \eqref{eq:tight}.

As a consequence, $\tilde\mu_t=(\sfe_t)_\sharp \eeta=:\mu_t$ for any $t \in [0,T]$ (cf. \cite[Lemma 5.2.1]{ags}), so that item (2) follows by Theorem \ref{thm:apriori-estimate}.

The last assertion is an immediate application of Proposition \ref{prop:compactness}.
    
\end{proof}

\section{Total dissipativity and lifting to curves}\label{sec:IE}
Let $(\Omega, \cB)$ be a standard Borel space endowed with a non atomic probability measure $\P$.
We recall that (cf.~\cite{CSS2grande}) every maximal totally $\lambda$-dissipative \MPVF $\frF\subset\prob_2(\TX)$ is in one-to-one correspondence with a maximal $\lambda$-dissipative (multivalued) operator $\Bb\subset L^2(\Omega, \cB, \P; \X)\times L^2(\Omega, \cB, \P; \X)$, called \emph{Lagrangian representation of $\frF$}. This fact revealed to be of particular importance in the context of (totally) dissipative \MPVF{s} in \cite{CSS2grande}, due the well-developed theory on monotone/dissipative operators in Hilbert spaces (cf. \cite{BC17,BrezisFR}), as it is $\Bb$.

Recall that there exists (cf.~\cite[Theorems A.5, A.6]{CSS2grande}) a semigroup of $e^{\lambda t}$-Lipschitz transformations $(\Sgp_t)_{t \ge
  0}$ with $\Sgp_t: \overline{\dom(\mmo)} \to
\overline{\dom(\mmo)}$ s.t.~for every $\bar X\in \dom(\mmo)$
the curve $t \mapsto \Sgp_t \bar X$ is included in $\dom(\mmo)$ and it
is the unique locally Lipschitz continuous solution of the differential inclusion
\begin{equation*}
\begin{cases}
  \dot{X}_t \in \mmo X_t \quad \text{ a.e. } t>0, \\
  X\restr{t=0} = \bar X.
\end{cases}
\end{equation*}
Given $D \subset \prob_2(\X)$, we set 
\begin{equation}\label{eq:setSxmu}
    \Sp{\X, D}:= \left \{ (x,\mu) \in \X \times D \mid x \in \supp(\mu) \right \}.
\end{equation}
By \cite[Theorem 3.4]{CSS2grande} or \cite[Theorem 4.12]{CSS2piccolo}), for every $t \ge 0$, there exists a uniquely defined continuous map $\ss_t:\Sp{\X,\overline{\dom(\frF)}}\to \X$ such that for every $\mu \in \overline{\dom(\frF)}$, the map $\ss_t(\cdot,\mu): \supp(\mu) \to \X$ is $e^{\lambda t}$-Lipschitz continuous and
\[
  \label{eq:7-2} \text{ for every $X \in \overline{\dom(\mmo)}$, } \Sgp_t X(\omega)=\ss_t(X(\omega),X_\sharp\P) \text{ for $\P$-a.e.~$\omega \in \Omega$.}
\]
The map $\ss_t$ is also associated to the semigroup of $e^{\lambda t}$-Lipschitz transformations $S_t$ of $\frF$ in $\prob_2(\TX)$ via the formula
$S_t(\bar\mu):=\ss_t(\cdot,\bar\mu)_\sharp{\bar\mu}$ (cf. \cite[Definition 4.1]{CSS2grande}), see the next Theorem \ref{thm:implift}.
Reasoning as in \cite[Eq. (4.14)]{CSS2grande}, if $\bar\mu\in\dom(\frF)$, we can associate to $(\ss_t)_{t\ge0}$ a $\bar\mu$-measurable map
\begin{equation}\label{eq:s}
    {\mathrm s}_{\bar\mu}:\X\to H^1(0,T;\X),\quad\mathrm s_{\bar\mu}[x](t):=\ss_t(x,\bar\mu).
\end{equation}
Notice also that, by definition of semigroup, whenever $\bar X \in L^2(\Omega, \cB, \P; \X)$ is such that $(\bar X)_\sharp \P=\bar\mu$, then
\begin{equation}\label{eq:init}
    \sfe_0 \circ {\mathrm s}_{\bar\mu} \circ \bar X = \bar X.
\end{equation}
The following result comes from \cite[Theorems 4.2, 4.4, 4.8]{CSS2grande}.

\begin{theorem}\label{thm:implift} Let $\frF \subset \prob_2(\TX)$ be a maximal totally $\lambda$-dissipative
    \MPVF. Then,
    for every $\bar\mu\in \overline{\dom(\frF)}$, the curve $\mu:[0,+\infty)\to\prob_2(\X)$, $\mu_t:=S_t(\bar\mu)$, is the unique $\lambda$-EVI solution in $[0,+\infty)$ for $\frF$ starting from 
    $\bar\mu$. Moreover
    $S_t$ is a semigroup of $e^{\lambda t}$-Lipschitz transformations
    satisfying
    \begin{equation*}
        W_2(S_t(\bar\mu'),S_t(\bar\mu''))
        \le e^{\lambda t}
        W_2(\bar\mu',\bar\mu'')
        \quad\text{for every }
        \bar\mu',\bar\mu''\in \overline{\dom(\frF)},\,t\ge0.
    \end{equation*}
    If $T>0$ and, in addition, $\bar\mu \in \dom(\frF)$, then $\eeta:=(\mathrm s_{\bar\mu})_\sharp \bar\mu\in\prob(\rmC([0,T];\X))$ is the unique element of $\prob(\rmC([0,T]; \X))$ concentrated on absolutely continuous curves satisfying 
    \begin{enumerate}
  \item $(\mathsf e_t)_\sharp \eeta=\mu_t$ for every $t\in [0,T]$;
  \item $\eeta$-a.e.~$\gamma$ is an integral solution of the
    differential equation
    $\dot\gamma(t)=\bb^\circ_t(\gamma(t))$ a.e.~in $[0,T]$, where $\bb^\circ$ is the map associated with the unique element of minimal norm of $\Bb$ as in \cite[Theorem 3.4]{CSS2grande}.
    \end{enumerate}
    Finally, if $\bar\mu\in \overline{\dom(\frF)}$ with $\#(\supp(\bar\mu))$ finite, then 
    \[\#(\supp(\mu_t)) \text{ is finite and non-increasing w.r.t.~} t\ge 0.\]
\end{theorem}

We also recall the following result taken from \cite[Theorem 3.20]{CSS2grande}.
\begin{theorem}[The minimal selection]\label{thm:minsel}
  Let $\frF\subset\prob_2(\TX)$ be a maximal totally $\lambda$-dissipative
  \MPVF. The following hold.
  \begin{enumerate}
  \item For every $\mu\in \dom(\frF)$ there exists a unique vector field
    $\ff^\circ[\mu]\in L^2(\X,\mu;\X)$ such that
    \begin{equation}\label{eq:defq}
      (\ii_\X, \ff^\circ[\mu])_\sharp\mu\in \frF[\mu],\quad
      \int |\ff^\circ[\mu]|^2\,\d\mu\le 
      \int |v|^2\,\d\Phi\quad\text{for every }\Phi\in\frF[\mu].
    \end{equation}
    We denote the \emph{minimal selection of $\frF$} at $\mu$ by
    \[\frF^\circ[\mu]:=(\ii_\X, \ff^\circ[\mu])_\sharp\mu.\]
     \item If $\mmo$ is the Lagrangian representation of $\frF$, then for
    every $\mu\in \dom(\frF)$, we have 
    \[\ff^\circ[\mu]=\bb^\circ[\mu]\quad\mu\text{-a.e.},\]
    where $\bb^\circ$ as in \cite[Theorem 3.4]{CSS2grande}.
    \item The map $|\frF|_2:\prob_2(\X)\to [0,+\infty]$ defined by
    \begin{equation}
      \label{eq:21}
      |\frF|_2(\mu):=
      \begin{cases}
        \displaystyle
        \int |\ff^\circ[\mu]|^2\,\d\mu&\text{if }\mu\in \dom(\frF),\\
        +\infty&\text{if }\mu \not\in \dom(\frF)
      \end{cases}
    \end{equation}
    is lower semicontinuous.
\end{enumerate}
\end{theorem}

The following stability result is a simple consequence of \cite[Theorem 4.9]{CSS2grande}.
\begin{proposition}\label{prop:yell} Let $\frF \subset \prob_2(\TX)$ be a maximal totally $\lambda$-dissipative \MPVF and let $\bar\mu_{h}, \bar\mu \in \dom(\frF)$, $h\in\N$, be such that \begin{enumerate}
        \item $(\bar\mu_{h})_{h}$ converges to $\bar\mu$ in $\prob_2(\X)$, as $h\to\infty$;
        \item $S:=\sup_h |\frF|_2(\bar\mu_{h})<\infty$.
    \end{enumerate}
If $T>0$ and $\eeta_h, \eeta \in \prob(\rmC([0,T]; \X))$ are as in Theorem \ref{thm:implift} for $\bar\mu_{h}$ and $\bar\mu$ respectively, then there exists $C>0$ depending only on $T,\lambda,S,\eeta$ such that
\[ W_{2, \infty} (\eeta_h, \eeta) \le C \, \sqrt{W_2(\bar\mu_{h}, \bar\mu)} \quad \text{ for every }h\in\N.\]
\end{proposition}
\begin{proof} Applying \cite[Proposition 3.18]{CSS2piccolo}, we can find random variables $\bar X, \bar X_h \in L^2(\Omega, \cB, \P; \X)$ such that 
    $(\bar X_h)_\sharp \P=\bar\mu_{h}$, $\bar X_\sharp \P =\bar\mu $, and $|\bar X_h - \bar X|_{L^2(\Omega; \X)} \le 2 \,W_2(\bar\mu_{h}, \bar\mu)$.
    We now consider the family of 
    $\P$-measurable maps 
    $\mathrm X_h:\Omega\to H^1(0,T;\X)
    \subset \mathrm C([0,T];\X)$ 
    defined as $\mathrm X_h:= \mathrm{s}_{\bar\mu_{h}} \circ \bar X_h$ and $\mathrm X:= \mathrm{s}_{\bar\mu} \circ \bar X$ (cf. \eqref{eq:s}). By the proof of \cite[Theorem 4.9]{CSS2grande}, we have
\begin{align*}
    \|\mathrm X_h-\mathrm X\|_{L^2(\Omega;L^2(0,T;\X))}^2 &\le T\,\mathrm e^{2\lambda_+ T} |\bar X_h-\bar X|^2_{L^2(\Omega; \X)},\\
    \|\mathrm X_h-\mathrm X\|^2_{L^2(\Omega;\mathrm C([0,T];\X))} &\le D\,\left(S'+\|\mathrm X\|_{L^2(\Omega; H^1(0,T;\X))}\right)
         \|\mathrm X_h-\mathrm X\|_{L^2(\Omega;L^2(0,T;\X))},
\end{align*}    
for constants $D,S'>0$ depending only on $T,\lambda,S$. The sought order of convergence immediately follows observing that 
\[ W_{2, \infty}(\eeta_h, \eeta) \le  \|\mathrm X_h-\mathrm X\|_{L^2(\Omega;\mathrm C([0,T];\X))}. \]
\end{proof}

\section{\MPVF{s} with totally dissipative barycenter}\label{sec:final} 
We first introduce another intermediate notion of dissipativity, that will play a crucial role.
\begin{definition}[Unconditional $\lambda$-dissipativity]
    \label{def:u-dissipativity}
    We say that a \MPVF $\frF \subset \prob_2(\TX)$ is \emph{unconditionally $\lambda$-dissipative}, $\lambda\in\R$, if 
    for every $\ggamma\in \prob_2(\X\times\X)$ and every $\Phi_i\in \frF[{\pi^i}_\sharp\ggamma]$, $i=0,1$, there exists $\ttheta\in \Gamma(\Phi_0,\Phi_1)$ such that $(\sfx^0,\sfx^1)_\sharp\ttheta=\ggamma$ and
    \begin{equation*}
        \int \langle v_1-v_0,x_1-x_0
        \rangle \,\d\ttheta(x_0,v_0,x_1,v_1)\le \lambda \int |x_1-x_0|^2\d\ggamma(x_0,x_1).
    \end{equation*}
\end{definition}
It is easy to check that total $\lambda$-dissipativity implies unconditional $\lambda$-dissipativity, which implies (metric) $\lambda$-dissipativity.
It is remarkable that a totally dissipative barycenter is enough to get unconditional dissipativity.

\begin{proposition} \label{prop:dissbari} Let $\frF \subset \prob_2(\TX)$ be a \MPVF whose barycenter $\bri\frF$ is totally $\lambda$-dissipative. Given $\Phi, \Psi \in \frF$ and $\ggamma \in \Gamma(\sfx_\sharp \Phi, \sfx_\sharp \Psi)$, define
\[ \ssigma := \int_{\X^2} \left ( \delta_{x} \otimes \Phi_{x} \right ) \otimes \left (  \delta_{y} \otimes \Psi_{y} \right ) \de \ggamma(x, y) \in \prob(\TX\times\TX),\]
where $\{\Phi_{x}\}_{x \in \X}$ and $\{\Psi_{y}\}_{y \in \X}$ are the disintegrations of $\Phi$ and $\Psi$ w.r.t.~their spatial marginals i.e.
\[ \Phi = \int_{\X} (\delta_{x} \otimes \Phi_{x}) \de (\sfx_\sharp \Phi)(x), \quad \Psi = \int_{\X} (\delta_{y} \otimes \Psi_{y}) \de (\sfx_\sharp \Psi)(y). \]
Then $\sigma\in\Gamma(\Phi, \Psi)$, $(\sfx^0, \sfx^1)_\sharp \ssigma = \ggamma$, and 
\begin{equation}\label{eq:dissbri}
    \int_{\TX} \langle x - y, v - w \rangle \de \ssigma (x, v, y, w) \le \lambda\int_{\X^2} |x-y|^2 \de \ggamma(x, y).
\end{equation}
In particular, $\frF$ is unconditionally $\lambda$-dissipative.
\end{proposition}
\begin{proof}
    It is clear that $\ssigma \in \Gamma(\Phi, \Psi)$ and that $(\sfx^0, \sfx^1)_\sharp \ssigma = \ggamma$. Let us show \eqref{eq:dissbri}:
\begin{equation}\label{eq:dissUD}
\begin{split}
\int_{\TX} \langle x - y, v - w \rangle \de \ssigma (x, v, y, w) &= \int_{\X^2} \int_{\X} \int_{\X}   \langle x - y, v - w \rangle \de \Phi_{x}(v) \de \Psi_{y}(w) \de \ggamma(x, y) \\
& = \int_{\X^2} \langle x-y, \bry\Phi(x)- \bry\Psi(y) \rangle \de \ggamma(x, y) \\
& \le  \lambda\int_{\X^2} |x-y|^2 \de \ggamma(x, y),
\end{split}
\end{equation}
by total $\lambda$-dissipativity of the barycenter.
\end{proof}

\begin{remark}\label{rem:riscr} For later use, we rewrite the expression \eqref{eq:dissbri} of  Proposition \ref{prop:dissbari}, which involves the plan $\ssigma\in\prob(\TX\times\TX)$, with a correspondent expression involving a plan in $\prob(\X^4)$. Let $\tau>0$, $\mu:=\sfx_\sharp \Phi$, $\nu:=\sfx_\sharp \Psi$, and set  $T_\tau\equiv T_\tau(\Phi)\in\prob(\X^2)$ and $\tilde{T}_\tau\equiv T_\tau(\Psi)\in\prob(\X^2)$ the single-step plans associated with $\Phi$ and $\Psi$, respectively, as in Definition \ref{def:step-plans}.
Define the plan
\[ \ttheta := \int_{\X^2} \left ( \delta_{x_0} \otimes T_{\tau,x_0} \right ) \otimes \left ( \delta_{y_0} \otimes \tilde{T}_{\tau,y_0} \right ) \de \ggamma(x_0,y_0) \in \prob(\X^4). \]
Then, under the assumptions of Proposition \ref{prop:dissbari}, we get
\[ \int_{\X^4} \langle x_0-y_0, \frac{x_1-x_0}{\tau} - \frac{y_1-y_0}{\tau} \rangle \de \ttheta(x_0, x_1, y_0, y_1) \le \lambda\int_{\X^2} |x_0-y_0|^2 \de \ggamma(x_0, y_0). \]
Indeed, we can prove it by disintegration arguments: denote by $\{\ggamma_{x_0}\}_{x_0 \in \X}$ and $\{\ggamma_{y_0}\}_{y_0 \in \X}$ the disintegrations of $\ggamma$ w.r.t.~the first and second marginals, respectively, i.e.
\[ \ggamma = \int_{\X} (\delta_{x_0} \otimes \ggamma_{x_0}) \de \mu(x_0), \quad \ggamma = \int_{\X} (\delta_{y_0} \otimes \ggamma_{y_0}) \de \nu(y_0). \]
Then
\begin{equation*}
\begin{split}
&\int_{\X^4} \langle x_0-y_0, \frac{x_1-x_0}{\tau} - \frac{y_1-y_0}{\tau} \rangle \de \ttheta(x_0, x_1, y_0, y_1)\\
&=\int_{\X^4} \langle x_0-y_0, \frac{x_1-x_0}{\tau} - \frac{y_1-y_0}{\tau} \rangle \de T_{\tau,x_0}(x_1)\de \tilde{T}_{\tau,y_0}(y_1)\de\ggamma(x_0,y_0)\\
&=\int_{\X^4} \langle x_0-y_0, \frac{x_1-x_0}{\tau} - \frac{y_1-y_0}{\tau} \rangle \de \tilde{T}_{\tau,y_0}(y_1)\de\ggamma_{x_0}(y_0)\de T_\tau(x_0,x_1)\\
&=\int_{\TX\times\X^2} \langle x_0-y_0, v_0 - \frac{y_1-y_0}{\tau} \rangle \de \tilde{T}_{\tau,y_0}(y_1)\de\ggamma_{x_0}(y_0)\de \Phi(x_0,v_0)\\
&=\int_{\X^4} \langle x_0-y_0, v_0 - \frac{y_1-y_0}{\tau} \rangle \de \tilde{T}_{\tau,y_0}(y_1)\de\ggamma_{x_0}(y_0)\de \Phi_{x_0}(v_0)\de\mu(x_0)\\
&=\int_{\X^4} \langle x_0-y_0, v_0 - \frac{y_1-y_0}{\tau} \rangle \de \tilde{T}_{\tau,y_0}(y_1)\de \Phi_{x_0}(v_0)\de\ggamma(x_0,y_0)\\
&=\int_{\X^4} \langle x_0-y_0, v_0 - \frac{y_1-y_0}{\tau} \rangle\de \Phi_{x_0}(v_0)\de\ggamma_{y_0}(x_0) \de \tilde{T}_\tau(y_0,y_1)\\
&=\int_{\TX\times\X^2} \langle x_0-y_0, v_0 - w_0 \rangle\de \Phi_{x_0}(v_0)\de\ggamma_{y_0}(x_0) \de \Psi(y_0,w_0)\\
&=\int_{\X^4} \langle x_0-y_0, v_0 - w_0 \rangle\de \Phi_{x_0}(v_0) \de \Psi_{y_0}(w_0)\de\ggamma(x_0,y_0),
\end{split}
\end{equation*}
and conclude as in \eqref{eq:dissUD}.
Otherwise, one can analogously prove that
\[\ttheta=\left(\sfx^0,\exp^\tau\circ(\sfx^0,\sfv^0),\sfx^1,\exp^\tau\circ(\sfx^1,\sfv^1)\right)_\sharp\ssigma.\]

\end{remark}
\medskip

In the following, we show that the $\lambda$-\EVI solution generated by  any maximal extension of the barycenter of a \MPVF $\frF$ is a $\lambda$-\EVI solution for the starting \MPVF $\frF$.  For the interested reader, we refer to \cite[Theorem 3.14]{CSS2grande} for conditions granting uniqueness of maximal extensions for a \MPVF.

\begin{proposition}\label{prop1} Let $\frF \subset \prob_2(\TX)$ be a \MPVF whose barycenter $\bri\frF$ is totally $\lambda$-dissipative. If $\bar\mu\in\overline{\dom(\frF)}$ and $\mu_t:= S_t(\bar\mu)$, $t \ge 0$, is as in Theorem \ref{thm:implift} for any maximal totally $\lambda$-dissipative extension of $\bri \frF$, then it is also a $\lambda$-\EVI solution for $\frF$.
\end{proposition}
\begin{proof}
Let $\Gg$ be  any maximal totally $\lambda$-dissipative extension of $\bri \frF$.
Notice that, since $\dom(\frF)=\dom(\bri{\frF})$, then $\overline{\dom(\frF)}\subset\overline{\dom(\Gg)}$. By Theorem \ref{thm:implift}, the curve $\mu$ is the unique $\lambda$-\EVI solution for $\Gg$ starting from $\bar\mu$. Thus, by Definition \ref{def:evi}, $\mu$ is also a $\lambda$-\EVI solution for $\bri\frF$. Let $t\in(0,+\infty)$ and $\Phi\in\frF$. Recalling that $\bri{\Phi}$ is concentrated on the map $\bry\Phi$ (cf. Definition \ref{def:baryc}), we have
 \begin{equation*}
    \begin{aligned}
      \frac12\frac\d{\d t}^+ W_2^2(\mu_t,\sfx_\sharp \Phi)
      &\le \lambda W_2^2(\mu_t,\sfx_\sharp \Phi)-
      \bram{\bri{\Phi}}{\mu_t}\\
      &= \lambda W_2^2(\mu_t,\sfx_\sharp \Phi)-\int_{\TX \times\X} \scalprod{\bry\Phi(x_0)}{x_0-x_1} \de \bar{\ggamma}(x_0,x_1),
    \end{aligned}
  \end{equation*}
  where
  \[\bar{\ggamma}=\argmin_{\ggamma\in\Gamma_o(\sfx_\sharp \Phi,\mu_t)}\int_{\TX \times\X} \scalprod{\bry\Phi(x_0)}{x_0-x_1} \de \ggamma(x_0,x_1).\]
  In order to prove that $\mu$ is a $\lambda$-\EVI for $\frF$, it is then sufficient to show that 
  \[\bram{\Phi}{\mu_t}\le \int_{\TX \times\X} \scalprod{\bry\Phi(x_0)}{x_0-x_1} \de \bar{\ggamma}(x_0,x_1).\]
  Let $\{\Phi_x\}_{x\in\X}\subset\prob_2(\X)$ be the disintegration of $\Phi$ w.r.t. the projection map $\sfx$ and $\ttheta\in\Lambda(\Phi,\mu_t)$ defined by $\ttheta:=\int \Phi_{x_0}(v_0)\,\d \bar{\ggamma}(x_0,x_1)$. We have
  \begin{align*}
  \bram{\Phi}{\mu_t}&\le \int_{\TX \times\X} \scalprod{v_0}{x_0-x_1} \de \ttheta(x_0,v_0,x_1)\\
  &=\int_{\TX} \scalprod{v_0}{x_0} \de \Phi(x_0,v_0)-\int_{\TX\times\X} \scalprod{v_0}{x_1} \de \ttheta(x_0,v_0,x_1)\\
  &=\int_{\X} \scalprod{\bry\Phi(x_0)}{x_0} \de [\sfx_\sharp\Phi](x_0)-\int_{\X\times\X} \scalprod{\bry\Phi(x_0)}{x_1} \de \bar{\ggamma}(x_0,x_1)\\
  &=\int_{\X\times\X} \scalprod{\bry\Phi(x_0)}{x_0-x_1} \de \bar{\ggamma}(x_0,x_1).
  \end{align*}
\end{proof}

We state the first convergence result of a sequence of probabilistic representations of the Explicit Euler scheme (corresponding to vanishing step sizes) for a metrically dissipative \MPVF $\frF$ towards the probabilistic representation of the Implicit Euler scheme for a maximal totally dissipative extension of the barycenter of $\frF$. This result, stated in Theorem \ref{thm:red}, is proved here assuming that the starting point $\bar\mu$ is finitely supported; we will extend it to a general $\bar\mu$ in Theorem \ref{thm:main}. To prove Theorem \ref{thm:red}, we apply the following Theorem \ref{thm:sticky1}, stated here in a simplified form and whose proof is postponed to Appendix \ref{sec:sticky}.

\begin{theorem}[Sticky particles representation]\label{thm:sticky1}
Let $\mu:[0,+\infty)\to\prob_2(\X)$ be such that
\begin{equation*}
    \#(\supp(\mu_t)) \text{ is finite and non-increasing w.r.t.~} t\ge 0.
\end{equation*}
Then there exists at most one $\eeta\in \prob(\rmC([0,+\infty);\X))$ such that $(\mathsf{e}_t)_\sharp \eeta=\mu_t$ for every $t\ge 0$.
\end{theorem}

\begin{theorem}\label{thm:red} Let $\frF \subset \prob_2(\TX)$ be a \MPVF whose barycenter $\bri\frF$ is totally $\lambda$-dissipative. Let $\bar\mu \in \dom(\frF)$ be a measure with finite support. Let $\N\ni h\mapsto \tau(h)\in\R_+$ be a vanishing sequence of
  time steps, and let $(M_h, \frF_h) \in \mathscr E(\bar \mu,\tau(h),T,L)$, $T, L>0$, be interpolation of the Explicit Euler scheme for $\frF$, for any $h\in\N$. Let
\begin{enumerate}
    \item $\eeta_h$ be satisfying \eqref{item1-etatau},\eqref{item2-etatau} of Proposition \ref{prop:exlift} for $(M_h, \frF_h)$;
    \item $\eeta$ be as in Theorem \ref{thm:implift} for any maximal totally $\lambda$-dissipative extension of $\bri \frF$.
\end{enumerate}
Then $W_{2, \infty}(\eeta_h, \eeta) \to 0$ as $h \to + \infty$.
\end{theorem}
\begin{proof}
    By Proposition \ref{prop:tight}, we know that, up to a subsequence, $\eeta_h$ converges as $h \to + \infty$ in $W_{2, \infty}$ to a measure $\tilde{\eeta} \in \prob_2(\rmC([0,T];\X)$ such that $(\sfe_t)_\sharp \tilde{\eeta}= \mu_t$ for every $t \in [0,T]$, where $\mu$  is the unique $\lambda$-\EVI solution in $[0,T]$ for $\frF$ starting from $\bar\mu$. 
    By Proposition \ref{prop1}, we deduce that $(\sfe_t)_\sharp \tilde{\eeta} = (\sfe_t)_\sharp \eeta$ for every $t \in [0,T]$. Moreover, by Theorem \ref{thm:implift}, we have that $\#( \supp(\mu_t))$ is finite and non-increasing, so that we conclude by Theorem \ref{thm:sticky1} that $\tilde{\eeta} = \eeta$. This also gives the convergence without passing to subsequences.
\end{proof}

\begin{theorem}\label{thm:giulia} Let $\frF\subset \prob_2(\TX)$ be a \MPVF whose barycenter $\bri\frF$ is totally $\lambda$-dissipative. Let $\tau, T, L>0$ be such that $\lambda_+\tau\le2$ and $\tau<1$. Let
\begin{enumerate}[label=(\roman*)]
\item $\bar\mu, \tilde\mu \in \dom(\frF)$ with $W_2(\bar\mu, \tilde{\mu}) <1$;
\item $(M_\tau, \frF_\tau,\eeta_\tau)\in \mathscr T(\bar\mu,\tau,T,L)$ and  $(\tilde M_\tau, \tilde\frF_\tau,\tilde{\eeta}_\tau)\in \mathscr T(\tilde\mu,\tau,T,L)$.
\end{enumerate}
Then there exists a constant $C$ (depending only on $L,T,\lambda$) such that
\[W_{2, \infty}(\eeta_\tau, \tilde\eeta_\tau)\le C \left(W_2^{\frac{1}{2}}(\bar\mu,\tilde\mu)+\tau^{\frac{1}{4}}\right).\]
\end{theorem}
\begin{proof}
We set
$N:=N(T,\tau)$ and, for $k=0, \dots, N-1$, we set $\Phi^k := \frF_\tau(\tau k)$, $\tilde\Phi^{k} := \tilde\frF_\tau(\tau k)$;  while for $k=0, \dots, N$, we set $\mu^k:= M_\tau(\tau k)$, and $\tilde{\mu}^k:= \tilde{M}_\tau(\tau k)$; we also fix $\rrho^0\in\Gamma_o(\bar\mu,\tilde\mu)$.  Following Definition \ref{def:step-plans}, we define the single-step plans $T_\tau^k:=T_\tau(\Phi^k)$, $\tilde{T}_\tau^k:=T_\tau(\tilde{\Phi}^k)$ and the multi-step plans $\aalpha_\tau:=\aalpha_\tau\left((\Phi^k)_{0\le k\le N-1}\right)\in \prob(\X^{N+1})$, $\tilde{\aalpha}_\tau:=\aalpha_\tau\left((\tilde\Phi^k)_{0\le k\le N-1}\right)\in \prob(\X^{N+1})$.
Recall that, by Definition \ref{def:intEE}, $\eeta_\tau = G_\sharp \alpha_\tau$, $\tilde \eeta_\tau = G_\sharp \tilde \alpha_\tau$.

For every $n \in \{1, \dots, N\}$, define $\proj_n^0,\proj_n^1:\X^{n+1} \times \X^{n+1}\to\X^{n+1}$ as the projection maps 
\[\proj_n^0(x_0,\dots, x_n, y_0, \dots, y_n):= (x_0, \dots, x_n), \quad\proj_n^1(x_0,\dots, x_n, y_0, \dots, y_n):= (y_0, \dots, y_n).\]
and the restriction maps $\Rest_n:\X^{N+1}\to\X^{n+1}$, $\Rest_n(x_0,\dots,x_N):=(x_0,\dots,x_n)$.
We start with the following claim. \\

\textit{Claim 1}:  for every $n \in \{1, \dots, N\}$, there exists $\ttheta^n \in \prob(\X^{n+1} \times \X^{n+1})$ such that 
\[({\proj_n^{0}})_\sharp \ttheta^n = (\Rest_n)_\sharp \aalpha_\tau,\quad ({\proj_n^{1}})_\sharp \ttheta^n = (\Rest_n)_\sharp \tilde{\aalpha}_\tau,\quad ({\proj_0^{0}},{\proj_0^{1}})_\sharp \ttheta^n=\rrho_0, \]
 and 
\begin{equation}\label{eq:dissii}
\int \langle x_k-y_k, \frac{x_{k+1}-x_k}{\tau} - \frac{y_{k+1}-y_k}{\tau} \rangle \de \ttheta^n \le \lambda\int_{\X^2} |x_k-y_k|^2 \de \ttheta^n \quad \text{ for every } k=0, \dots, n-1.
\end{equation}
\textit{Proof of claim 1}: we prove it by induction, starting from $n=1$. In this case it is sufficient to define
\[ \ttheta^1 := \int \left ( \delta_{x_0} \otimes T^0_{\tau, x_0}  \right ) \otimes \left ( \delta_{y_0} \otimes \tilde{T}^0_{\tau, y_0} \right ) \de \rrho^0(x_0, y_0)\in \prob(\X^{2} \times \X^{2})\]
and use Remark \ref{rem:riscr}.  Notice that $({\proj_1^{0}})_\sharp \ttheta^1 = T_\tau^0=(\Rest_1)_\sharp \aalpha_\tau$ and $({\proj_1^1})_\sharp \ttheta^1 = \tilde{T}_\tau^0=(\Rest_1)_\sharp \tilde\aalpha_\tau$, by Lemma \ref{lem:Ralpha}.
Assume the statement to be true for some $n \in \{1 \dots, N-1\}$ and define
\[ \ttheta^{n+1}:= \int \left ( \delta_{x_0, \dots, x_{n}} \otimes T^n_{\tau, x_n}  \right ) \otimes \left ( \delta_{y_0, \dots, y_n} \otimes \tilde{T}^n_{\tau, y_n} \right ) \de \ttheta^n(x_0, \dots, x_n, y_0, \dots, y_n)\in\prob(\X^{n+2}\times\X^{n+2}).\]
By construction,  recalling Definition \ref{def:step-plans} of $\aalpha_\tau$ and $\tilde \aalpha_\tau$ and Lemma \ref{lem:Ralpha}, we have that $({\proj_{n+1}^{0}})_\sharp \ttheta^{n+1}=\aalpha_\tau^{n+1}=(\Rest_{n+1})_\sharp\aalpha_\tau$ and $({\proj_{n+1}^{1}})_\sharp \ttheta^{n+1}=\tilde{\aalpha}_\tau^{n+1}=(\Rest_{n+1})_\sharp\tilde{\aalpha}_\tau$.
If $k \in \{0, \dots, n-1\}$,  the validity of \eqref{eq:dissii} for $\ttheta^{n+1}$ and the fact that $({\proj_0^{0}},{\proj_0^{1}})_\sharp \ttheta^{n+1}=\rrho^0$ follow by the induction step and the fact that  the restriction from $\X^{n+2}\times\X^{n+2}$ to $\X^{n+1}\times \X^{n+1}$ of $\ttheta^{n+1}$, i.e. $(\Rest_n\times\Rest_n)_\sharp\ttheta^{n+1}$, is equal to $\ttheta^n$. If $k=n$, then 
\begin{align*}
    &\int \langle x_n-y_n, \frac{x_{n+1}-x_n}{\tau} - \frac{y_{n+1}-y_n}{\tau} \rangle \de \ttheta^{n+1}(x_0, \dots, x_{n+1}, y_0, \dots, y_{n+1}) \\
    &= \int \langle x_n-y_n, \frac{x_{n+1}-x_n}{\tau} - \frac{y_{n+1}-y_n}{\tau} \rangle \de T^n_{\tau, x_n}(x_{n+1}) \de \tilde T^n_{\tau, y_n}(y_{n+1}) \de \rrho^n(x_n, y_n),
\end{align*}
being $\rrho^n$ the marginal of $\ttheta^n$ in the variables $(x_n, y_n)$. Hence, \eqref{eq:dissii} follows again by applying Remark \ref{rem:riscr}. This concludes the induction step and the proof of the claim.

\medskip

We consider the map $J: \X^{N+1} \times \X^{N+1} \to \rmC([0,T]; \X) \times \rmC([0,T]; \X) $ defined as
\[ J(x_0, \dots, x_n, y_0, \dots, y_n):= (G(x_0, \dots, x_n), G(y_0, \dots, y_n)),\]
being $G$ defined as in Definition \ref{def:intEE}. Let us set $\Llambda := J_\sharp \ttheta^N$; using Claim 1, it is immediate to check that $\Llambda \in \Gamma(\eeta_\tau, \tilde{\eeta}_\tau)$ and 
\begin{equation}\label{eq:finaldiss}
\begin{split}
&\int \langle \sfe_{k\tau}(\gamma)-\sfe_{k\tau}(\tilde \gamma), \frac{\sfe_{(k+1)\tau}(\gamma)-\sfe_{k\tau}(\gamma)}{\tau} - \frac{\sfe_{(k+1)\tau}(\tilde\gamma)-\sfe_{k\tau}(\tilde \gamma)}{\tau} \rangle \de \Llambda(\gamma, \tilde \gamma) \\
& \quad \le \lambda\int |\sfe_{k\tau}(\gamma)-\sfe_{k\tau}(\tilde \gamma)|^2 \de \Llambda(\gamma, \tilde \gamma) 
\end{split}
\end{equation}
for every $k=0, \dots, N-1$. Let us set 
\[\sigma^2(t):=\displaystyle\int |\sfe_t(\gamma)-\sfe_t(\tilde\gamma)|^2\d\Llambda(\gamma,\tilde\gamma)=\|\sfe_t\circ\pi^0-\sfe_t\circ\pi^1\|^2_{\Llambda}, \quad t \in [0,T],\]
 where, in this context, $\pi^0$ and $\pi^1$ are defined on $\rmC([0,T]; \X)\times\rmC([0,T]; \X)$ and are the projection maps on the first and second component, respectively.
We then get, for any $k=0,\dots, N-1$,
\begin{align*}
\frac{\d }{\d t}\sigma^2(t)\bigg|_{t=k\tau+}&=\lim_{h\to0^+}\frac{1}{h}\,\int \left(|\sfe_{k\tau+h}(\gamma)-\sfe_{k\tau+h}(\tilde\gamma)|^2-|\sfe_{k\tau}(\gamma)-\sfe_{k\tau}(\tilde\gamma)|^2\right)\d\Llambda(\gamma,\tilde\gamma)\\
&=2\,\int \langle \sfe_{k\tau}(\gamma)-\sfe_{k\tau}(\tilde \gamma), \frac{\sfe_{(k+1)\tau}(\gamma)-\sfe_{k\tau}(\gamma)}{\tau} - \frac{\sfe_{(k+1)\tau}(\tilde\gamma)-\sfe_{k\tau}(\tilde \gamma)}{\tau} \rangle \de \Llambda(\gamma, \tilde \gamma)\\
& \quad + \lim_{h \to 0^+} h \int \left | \frac{\sfe_{(k+1)\tau}(\gamma)-\sfe_{k\tau}(\gamma)}{\tau} - \frac{\sfe_{(k+1)\tau}(\tilde\gamma)-\sfe_{k\tau}(\tilde \gamma)}{\tau} \right |^2 \de \Llambda(\gamma, \tilde \gamma) \\
& \le 2\lambda  \int |\sfe_{k\tau}(\gamma)-\sfe_{k\tau}(\tilde\gamma)|^2\d\Llambda(\gamma,\tilde\gamma) \\
& = 2\lambda \sigma^2(k\tau),
\end{align*}
where we have used \eqref{eq:finaldiss} and the fact that $\Llambda$ is concentrated on pair of curves which are affine in every interval $[k\tau, (k+1)\tau] \cap [0,T]$, $k=0, \dots, N-1$.

Arguing as in \cite[Proposition 3.4(2)]{CSS}, we notice that in every interval $[k\tau,(k+1)\tau] \cap [0,T]$ the function $t\mapsto \sigma^2(t)-4L^2 (t-k\tau)^2$ is concave. We then obtain
\begin{equation}\label{eq:almost-GW}
\frac{\d }{\d t}\sigma^2(t)\le 2\lambda \sigma^2(k\tau)+8L^2\tau,
\end{equation}
for every $t\in[0,T]$, with possible countable exceptions. We now proceed as in \cite[proof of Proposition 6.3]{CSS}, noting that, whenever $t \in [k \tau, (k+1)\tau) \cap [0,T]$ for some $k=0, \dots, N-1$, we have
\begin{align*}
&|\sigma(t)-\sigma(k\tau)|\\
&=\left|\|\sfe_t\circ\pi^0-\sfe_t\circ\pi^1\|_{\Llambda}-\|\sfe_{k\tau}\circ\pi^0-\sfe_{k\tau}\circ\pi^1\|_{\Llambda}\right|\\
&=\left|\|\sfe_t\circ\pi^0-\sfe_t\circ\pi^1\|_{\Llambda}-\|\sfe_t\circ\pi^1-\sfe_{k\tau}\circ\pi^0\|_{\Llambda}+\|\sfe_t\circ\pi^1-\sfe_{k\tau}\circ\pi^0\|_{\Llambda}-\|\sfe_{k\tau}\circ\pi^0-\sfe_{k\tau}\circ\pi^1\|_{\Llambda}\right|\\
&\le \|\sfe_t\circ\pi^0-\sfe_{k\tau}\circ\pi^0\|_{\Llambda}+\|\sfe_t\circ\pi^1-\sfe_{k\tau}\circ\pi^1\|_{\Llambda}\\
&=\left(\int|\gamma(t)-\gamma(k\tau)|^2\d\eeta_\tau(\gamma)\right)^{\frac{1}{2}}+\left(\int|\tilde\gamma(t)-\tilde\gamma(k\tau)|^2\d\tilde\eeta_\tau(\tilde\gamma)\right)^{\frac{1}{2}}\\
&\le 2L\tau.
\end{align*}
Hence, by applying the general identity
\[a^2-b^2=2b(a-b)+|a-b|^2, \quad a,b \in \R,\]
in the relation obtained in \eqref{eq:almost-GW}, with $a:=\sigma(k\tau)$, $b:=\sigma(t)$, and using the assumption $\lambda_+\tau\le2$, we get
\[\frac{\d }{\d t}\sigma^2(t)\le 2\lambda\sigma^2(t)+8|\lambda|L\tau\sigma(t)+24L^2\tau,\]
for every $t \in [0,T]$ with at most countable exceptions.  Applying the Gronwall estimate in \cite[Lemma B.1 and equation (6.8)]{CSS}, we get
\[\sigma(t)\le\sigma(0)e^{\lambda t}+8L\sqrt{t\tau}\left(1+|\lambda|\sqrt{t\tau}\right)e^{\lambda_+t},\quad\text{for any }t\in[0,T].\]
In particular,
\begin{equation}
\sup_{t\in[0,T]}\sigma(t)\le \sigma(0)e^{\lambda_+ T}+8L\sqrt{T\tau}\left(1+|\lambda|\sqrt{T\tau}\right)e^{\lambda_+ T}=:K_\tau.
\end{equation}
We recall that, by construction of $\Llambda$, $\sigma(0)=W_2(\bar\mu,\tilde\mu)$, indeed we chose $\rrho^0\in\Gamma_o(\bar\mu,\tilde\mu)$.

Using the interpolation inequality (cf.~\cite[p.233 (iii)]{Breziss})
    \begin{displaymath}
        \|Y\|^2_\infty
        \le C\|Y\|_{L^2(0,T;\X)}\|Y\|_{H^1(0,T;\X)}
        \quad\text{for every }Y\in H^1(0,T;\X),\text{ for some }C>0,
    \end{displaymath}
we have that there exists $C>0$ such that
\begin{equation}\label{eq:I1I2}
\begin{split}
W_{2, \infty}^2(\eeta_\tau, \tilde\eeta_\tau)&\le \int \|\gamma-\tilde\gamma\|^2_\infty\d\Llambda(\gamma,\tilde\gamma)\\
&\le C \int \|\gamma-\tilde\gamma\|_{L^2(0,T;\X)}\,\|\gamma-\tilde\gamma\|_{H^1(0,T;\X)}\d\Llambda(\gamma,\tilde\gamma)\\
&\le C \left(\int \|\gamma-\tilde\gamma\|^2_{L^2(0,T;\X)}\d\Llambda(\gamma,\tilde\gamma)\right)^{\frac{1}{2}}\cdot \left(\int \|\gamma-\tilde\gamma\|^2_{H^1(0,T;\X)}\d\Llambda(\gamma,\tilde\gamma)\right)^{\frac{1}{2}},
\end{split}
\end{equation}
indeed $\Llambda\in\Gamma(\eeta_\tau,\tilde\eeta_\tau)$.

It remains to estimate the two terms at the right-hand side. Concerning the first term, notice that
\begin{align*}
I_1^2&:=\int \|\gamma-\tilde\gamma\|^2_{L^2(0,T;\X)}\d\Llambda(\gamma,\tilde\gamma)\\
&=\int \left(\int_0^T|\gamma(t)-\tilde\gamma(t)|^2\d t\right)\d\Llambda(\gamma,\tilde\gamma)\\
&=\int_0^T \left(\int|\gamma(t)-\tilde\gamma(t)|^2\d\Llambda(\gamma,\tilde\gamma)\right)\d t\\
&\le T\sup_{t\in[0,T]}\sigma^2(t)\\
&\le T K_\tau^2
\end{align*}
Regarding the second term in \eqref{eq:I1I2}, we get
\begin{align*}
I_2^2&:=\int \|\gamma-\tilde\gamma\|^2_{H^1(0,T;\X)}\d\Llambda(\gamma,\tilde\gamma)\\
&\le 2\left(I_1^2+\int\left(\int_0^T |\gamma'(t)-\tilde\gamma'(t)|^2\d t\right)\d\Llambda(\gamma,\tilde\gamma)\right)\\
&=2\left(I_1^2+\int_0^T\left(\int |\gamma'(t)-\tilde\gamma'(t)|^2\d\Llambda(\gamma,\tilde\gamma)\right)\d t\right)\\
&\le 2(I_1^2+4TL^2).
\end{align*}
In conclusion,
\begin{align*}
W_{2, \infty}^2(\eeta_\tau, \tilde\eeta_\tau)&\le \sqrt{2}\,C\, T^{\frac{1}{2}}\,K_\tau\, \sqrt{TK_\tau^2+4TL^2}\\
&\le \tilde C\,\left(\sigma(0)+\sigma^2(0)+\tau^{\frac{1}{2}}+\tau+\tau^2\right),
\end{align*}
for some $\tilde C>0$ depending on $L,T,\lambda$. In particular, if $\tau<1$ and $\sigma(0)<1$, we get
\[W_{2, \infty}(\eeta_\tau, \tilde\eeta_\tau)\le \hat C \left(W_2^{\frac{1}{2}}(\bar\mu,\tilde\mu)+\tau^{\frac{1}{4}}\right),\]
for some $\hat C>0$ depending on $L,T,\lambda$.
\end{proof}
In order to state our main result, we 
introduce a technical property.

\begin{definition}\label{def:locsolv} Let $\frF$ be a \MPVF, $T>0$ and $\bar\mu \in \dom(\frF)$. We say that the Explicit Euler scheme is \emph{approximately solvable at $\bar\mu$ up to time $T$} if there exist $\bar{\tau}, L>0$ and a sequence $(\bar\mu^{j})_{j\in\N} \subset \dom(\frF)$ such that
\begin{enumerate}[label=(\roman*)]
    \item $W_2(\bar\mu^{j}, \bar\mu) \to 0$ as $j \to +\infty$;
    \item $\supp(\bar\mu^{j})$ is finite  for every $j\in\N$;
    \item the sets $\mathscr E( \bar\mu^{j}, \tau, T,L)$ are not empty for every $0<\tau<\bar{\tau}$, for every $j\in\N$.
\end{enumerate}
\end{definition}
We conclude with the main result of the paper.
\begin{theorem}\label{thm:main} Let $\frF \subset \prob_2(\TX)$ be a \MPVF whose barycenter $\bri \frF$ is totally $\lambda$-dissipative. Let $\bar\mu \in \dom(\frF)$ and $(\bar\mu_h)_h \subset \dom(\frF)$ be such that $W_2(\bar\mu_h, \bar\mu) \to 0$ as $h \to + \infty$, and let $\N\ni h\mapsto \tau(h)\in\R_+$ be a vanishing sequence of
  time steps. Let $T, L>0$ and
\begin{enumerate}
\item  $(M_h, \frF_h,\eeta_h)\in\mathscr T(\bar\mu_h,\tau(h),T,L)$ be  as in Definition \ref{def:intEE} for the \MPVF $\frF$;
    \item $\eeta \in \prob_2(\rmC([0,T]; \X))$ be the probability measure given by Theorem \ref{thm:implift} for any maximal totally $\lambda$-dissipative extension $\hat\frF$ of $\bri\frF$, starting from $\bar\mu$.
\end{enumerate}
Assume in addition that the Explicit Euler scheme for $\frF$ is approximately solvable at $\bar\mu$ up to time $T$, according to Definition \ref{def:locsolv}.
Then $W_{2, \infty}(\eeta, \eeta_h) \to 0$ as $h \to + \infty$.
\end{theorem}
\begin{proof} Let $(\bar\mu^{j})_{j}  \subset \dom(\frF)$ be as in Definition \ref{def:locsolv} for $\bar{\mu}$. By assumption, there exists $\hat L>0$ such that $\mathscr E( \bar\mu^{j}, \tau(h), T,\hat L)\neq\emptyset$ for $h$ sufficiently large. Without loss of generality, we can assume that $\hat L=L$. By this condition, for any $j\in\N$ there exists $\bar \Phi^j\in\frF[\bar\mu^{j}]$ such that $|\bar \Phi^j|_2\le L$.

Also, 
we can assume w.l.o.g.~that $W_2(\bar\mu^{j}, \bar\mu_h) <1$ and $\tau(h)<\min\{\frac{2}{\lambda_+},1\}$ for every $j, h \in \N$.

Let  $(M_h^j, \frF_h^j,\eeta_h^j)\in\mathscr T(\bar\mu^{j},\tau(h),T,L)$.
By Theorem \ref{thm:giulia}, there exists $C>0$ such that\[W_{2, \infty}(\eeta_h,\eeta_h^j)\le C \left(\sqrt{W_2(\bar\mu_h, \bar\mu^{j})}+\tau(h)^{\frac{1}{4}}\right) \quad \text{ for every } j, h \in \N.\]
Let $\eeta^j\in\prob(\rmC([0,T]; \X))$ be as in Theorem \ref{thm:implift} for $\hat\frF$ and the starting point $\bar\mu^{j}$.  By assumption, we have
\[\sup_j |\hat\frF|_2(\bar\mu^j)<\infty,\]
indeed, for any $j\in\N$, denoting by $\hat\ff^\circ$ the minimal selection of $\hat\frF$ as in Theorem \ref{thm:minsel}, we have
\begin{equation}\label{eq:minselN}
\int \left|\hat\ff^\circ[\bar\mu^{j}](x)\right|^2\,\d\bar\mu^{j}(x)\le \int |v|^2\,\d\hat\Phi(x,v),
\end{equation}
for any $\hat\Phi\in\hat\frF[\bar\mu^{j}]$. In particular, \eqref{eq:minselN} holds for any $\hat\Phi\in\bri\frF[\bar\mu^{j}]$, and thus for $\bri{\bar\Phi^j}$, so that
\begin{equation*}
\int \left|\hat\ff^\circ[\bar\mu^{j}](x)\right|^2\,\d\bar\mu^{j}(x)\le \int |\bry{\bar\Phi^j}(x)|^2\,\d\bar\mu^{j}(x)\le \int |v|^2\,\d\bar\Phi^j(x,v)\le L.
\end{equation*}
We can thus apply Proposition \ref{prop:yell} for $\hat\frF$ and get the existence of $\tilde C>0$ such that
\[W_{2, \infty}(\eeta^j, \eeta)\le \tilde C\, \sqrt{W_2(\bar\mu, \bar\mu^{j})} \quad \text{ for every } j \in \N. \]
Thus, by triangle inequality we can write, for every $n, N \in \N$, that
\begin{align*}
    W_{2, \infty}(\eeta_h, \eeta) &\le W_{2, \infty}(\eeta_h, \eeta_h^j) + W_{2, \infty}(\eeta_h^j, \eeta^j) + W_{2, \infty}(\eeta^j, \eeta) \\
    & \le C \left(\sqrt{W_2(\bar\mu_h, \bar\mu^{j})}+\tau(h)^{\frac{1}{4}}\right) +  W_{2, \infty}(\eeta_h^j, \eeta^j) + \tilde C\, \sqrt{W_2(\bar\mu, \bar\mu^{j})}.
\end{align*}
The thesis follows by passing first to the limit as $h \to + \infty$ and applying Theorem \ref{thm:red} to $\eeta_h^j$ and $\eeta^j$, and then passing to the limit as $j \to + \infty$.
\end{proof}

\begin{theorem}\label{prop:young}
    In the same hypotheses of Theorem \ref{thm:main}, let $(\Omega, \cB)$ be a standard Borel space endowed with a non atomic probability measure $\P$. Let $\bar X,\bar X_h \in L^2(\Omega, \cB, \P; \X)$ be random variables such that $(\bar X)_\sharp \P=\bar\mu$, $(\bar X_h)_\sharp \P= \bar\mu_h$ for every $h \in \N$, and let $\mathrm Z_h \in L^2(\Omega, \cB, \P; \rmC([0,T]; \X))$ be such that $(\mathrm Z_h)_\sharp \P= \eeta_h$ for every $h \in \N$. If $\bar X_h \to \bar X$ $\P$-a.s.~and $\sfe_0 \circ \mathrm Z_h = \bar X_h$ for every $h \in \N$, then 
    \[ \mathrm Z_h \to {\mathrm s}_{\bar\mu} \circ \bar X \text{ in } L^2(\Omega, \cB, \P; \rmC([0,T]; \X)) \text{ as } h \to + \infty, \]
    where ${\mathrm s}_{\bar\mu}$ is as in \eqref{eq:s}.
\end{theorem}
 \begin{proof}  We set $\ttheta_h := (\ii_\Omega, \bar X_h, \mathrm Z_h)_\sharp \P$ and $\ttheta := (\ii_\Omega, \bar X, {\mathrm s}_{\bar\mu} \circ \bar X)_\sharp \P$; we claim that
 \begin{equation}\label{eq:thetaconv}
     \ttheta_h \weakto \ttheta \text{ as } h \to + \infty.
 \end{equation}
First of all, notice that $(\ttheta_h)_h$ is tight, since its marginals are tight: the first one is the constant $\P$, the second one is $\bar\mu_h$ which is converging to $\bar\mu$, and the third one is $\eeta_h$ which is converging to $\eeta$ by Theorem \ref{thm:main}. Up to a (unrelabeled) subsequence, we thus get that $\ttheta_h \weakto \bar \ttheta$ for some $\bar\ttheta \in \prob(\Omega \times \X \times \rmC([0,T];\X))$. Observe that $(\pi^{0,1})_\sharp \ttheta_h = (\ii_\Omega, \bar X_h)_\sharp \P \weakto (\ii_\Omega, \bar X)_\sharp \P=(\pi^{0,1})_\sharp \bar\ttheta$ by the Dominated Convergence Theorem; we also have that
\[(\pi^{1,2})_\sharp \ttheta_h = (\bar X_h, \mathrm Z_h)_\sharp \P = (\sfe_0, \ii_{\rmC})_\sharp \eeta_h \weakto (\sfe_0, \ii_{\rmC})_\sharp \eeta = (\bar X, {\mathrm s}_{\bar\mu} \circ \bar X)_\sharp \P = (\pi^{1,2})_\sharp \bar \ttheta,\]
where we have denoted by $\ii_{\rmC}$ the identity map in $\rmC([0,T];\X)$ and we have used that $\sfe_0 \circ {\mathrm s}_{\bar\mu} \circ \bar X=\bar X$ (cf.~\eqref{eq:init}), together with the fact that $\eeta= ({\mathrm s}_{\bar\mu})_\sharp \bar\mu$ coming from Theorem \ref{thm:implift}. Since $(\pi^{1,2})_\sharp \bar \ttheta $ is concentrated on a map, we deduce by \cite[Lemma 5.3.2]{ags} that it must be $\bar\ttheta = \ttheta$; this proves \eqref{eq:thetaconv}. An application of \cite[Lemma 5.4.1]{ags} gives that $(\bar X_h, \mathrm Z_h) \to (\bar X, {\mathrm s}_{\bar\mu} \circ \bar X)$ in probability, so that, in particular, $\mathrm Z_h \to {\mathrm s}_{\bar\mu} \circ \bar X$ in probability. However, since the measures $(\mathrm Z_h)_\sharp \P= \eeta_h$ are converging to $\eeta$ in $W_{2, \infty}$ by Theorem \ref{thm:main}, the random variables $(\mathrm Z_h)_h$ have uniformly integrable $2$-moments, and therefore they converge to ${\mathrm s}_{\bar\mu}$ in $L^2(\Omega, \cB, \P; \rmC([0,T]; \X))$.
     
 \end{proof}

\section{Examples}\label{sec:examples}
We present examples illustrating the application of the theory developed in this paper. In particular, we examine how the results established in Theorem \ref{thm:main} and Theorem \ref{prop:young} can be used in the setting of Stochastic Gradient-type evolutions—more broadly, within the framework of Stochastic Dissipative Flow theory—as well as in the context of dissipative evolutions governed by interaction fields.

\subsection{Stochastic Dissipative Flow}\label{ex:sgd1}
Given a vector field $b:\X\to\X$ and an initial datum $\bar{x}\in\X$, we consider the following deterministic Cauchy problem
\begin{equation}\label{eq:odeex1}
\dot{x}(t) = b(x(t)), \quad t\in [0,T], 
\quad x(0) = \bar{x}.
\end{equation}
We assume $b$ to be continuous and $\lambda$-dissipative for some $\lambda\in\R$, i.e.
\begin{equation}\label{eq:dissg}
\langle x_1-x_0,b(x_1)-b(x_0)\rangle\le\lambda |x_1-x_0|^2,\quad\text{for any }x_0,\,x_1\in\X,
\end{equation}
so that \eqref{eq:odeex1} is well-posed.

We assume that $b$ arises as a stochastic average of a family of Borel vector fields $g:\X\times\U\to\X$, where $(\U,\UU)$ is a probability space endowed with a non-atomic probability measure, i.e.
\[b(x)=\int_\U g(x,u)\,\d\UU(u), \quad x \in \X.\]
We require that there exists $L>0$ such that
\begin{equation}\label{eq:Lgstoch}
\int_\U |g(x,u)|^2\,\d\UU(u) \le L(1+|x|^2) \quad \text{ for every } x \in \X.
\end{equation}

Similarly to what is done for the \emph{stochastic gradient descent}, we consider the following stochastic scheme. 

We adopt the same notation of Definition \ref{def:intEE} for the interpolation map $G:\X^{\finalstep T\tau+1}\to\rmC([0,T]; \X)$ sending points $(x_0, x_1, \dots, x_\finalstep T\tau)$ to their affine interpolation, i.e. to the curve $\gamma$ given by
    \begin{equation}\label{eq:glueG}
     \gamma(t):= \frac{1}{\tau}\left((n+1)\tau-t\right) x_n+\frac{1}{\tau}(t-n\tau) x_{n+1},
     \end{equation}
if $t \in [n\tau, (n+1)\tau]\cap[0,T]$, $0\le n<\finalstep T\tau$.

\begin{definition}[SDF]\label{def:SDF} Let $(\Omega, \cB, \P)$ be a standard Borel probability space and $T>0$. Define $J:=\left\{\frac{T}{N}\,:\,N\in\N\setminus\{0\}\right\}$. We say that a family of maps $X_\tau:\Omega\to \rmC([0,T];\X)$, $\tau \in J$, is a \emph{Stochastic Dissipative Flow (\SDF) for $g$} if there exist random variables $(X^n_\tau)_{0 \le n \le N, \tau \in J} \subset L^2(\Omega; \X)$ and $V^k : \Omega \to \U$, $k \in \N$ such that
\begin{itemize}
    \item $(X^{0}_\tau)_{\tau \in J}$ and $(V^k)_k$ are independent,
    \item $(V^k)_\sharp \P= \UU$ for every $k \in \N$,
    \item $X^{n+1}_\tau=X^n_\tau+\tau\, g(X^n_\tau,V^n)$ for every $0\le n\le N-1$, $\tau \in J$,
\end{itemize}
and 
\begin{equation}\label{eq:SDFjoint}
X_\tau=G_\sharp\left(X^0_\tau,X^1_\tau,\dots,X^n_\tau\right)\quad \text{ for every } \tau \in J.
\end{equation}
\end{definition}

We point out that in Definition \ref{def:SDF} we assumed $J \subset \R$ to be countable in order to ensure the existence of a \SDF which can be proven via the Kolmogorov existence theorem (cf.~\cite[Theorem 36.1]{billingsley}).

We have the following result.
\begin{corollary}\label{cor:6.1} In the setting of Definition \ref{def:SDF}, assume that $(X^0_\tau)_{\tau\in J}$ converges $\P$-a.s.~to some $\bar X \in L^2(\Omega, \cB, \P; \X)$ and that $W_2(({X^0_\tau})_\sharp\P,({\bar X})_\sharp\P)\to 0$ as $\tau\downarrow 0$.
Then $X_\tau$ converges in $L^2(\Omega, \cB, \P; \rmC([0,T]; \X))$ as $\tau \downarrow 0$ to the unique solution of the deterministic ODE
\[\dot{X}_t=b(X_t),\quad t \in [0,T],\quad X_{t=0}=\bar X.\]
\end{corollary}
The result can be proved using the machinery developed in this paper as follows. We define the \PVF $\frF\subset\prob_2(\TX)$ by
\begin{equation}\label{eq:F-SDF}
\frF[\mu]:=(\pi^0,g)_\sharp(\mu\otimes\UU),\quad\mu\in\prob_2(\X).
\end{equation}
Notice that, for any $\mu\in\dom(\frF)=\prob_2(\X)$, the Borel family of probability measures $\left\{\frF_x[\mu]\right\}_{x\in\X}\subset\prob_2(\X)$, obtained by disintegrating $\frF[\mu]$ w.r.t. $\sfx_\sharp\frF=\mu$, is given by
\[\frF_x[\mu]= g(x,\cdot)_\sharp\UU,\quad\text{for }\mu\text{-a.e. }x\in\X.\]
In particular, for any $\mu\in\prob_2(\X)$, the barycenter
of $\frF[\mu]$ defined in \eqref{eq:bry} is given by
\[\int_\X v\,\d\frF_x[\mu](v)=b(x), \quad\text{for }\mu\text{-a.e. }x\in\X.\]
Thus, according to Definition \ref{def:baryc}, the \PVF $\bri\frF$ is given by
\begin{equation}\label{eq:brFex1}
\bri\frF[\mu]=(\ii_\X, b)_\sharp\mu,\quad\mu\in\prob_2(\X).
\end{equation}
We now show that $\bri\frF$ is \emph{totally $\lambda$-dissipative}, so that, thanks to Proposition \ref{prop:dissbari}, $\frF$ is unconditionally $\lambda$-dissipative. 
\begin{proposition}
The \PVF $\bri\frF$ defined in \eqref{eq:brFex1} is totally $\lambda$-dissipative.
\end{proposition}
\begin{proof}
Take $\mu,\nu\in\prob_2(\X)$ and $\boldsymbol\vartheta\in\Gamma(\bri{\frF}[\mu],\bri{\frF}[\nu])$. Since $\bri\frF[\mu],\bri\frF[\nu]$ are concentrated on maps, there exists $\ggamma\in\Gamma(\mu,\nu)$ such that
\[\boldsymbol\vartheta=\left(\pi^0,b\circ\pi^0,\pi^1,b\circ\pi^1\right)_\sharp\ggamma.\]
By assumption \eqref{eq:dissg}, we conclude that
\begin{align*}
\int_{\TX^2} \langle v_1-v_0,x_1-x_0\rangle \,\d\boldsymbol\vartheta(x_0,v_0,x_1,v_1)&=
\int_{\X^2}\langle b(x_1)-b(x_0),x_1-x_0\rangle \,\d\ggamma(x_0,x_1)\\
&\le\lambda\int_{\X^2}|x_1-x_0|^2\d\ggamma(x_0,x_1).
\end{align*}
\end{proof}

Thanks to \cite[Theorem 3.24]{CSS2grande}, since $b$ is continuous then $\bri\frF$ is maximal and 
\begin{equation}\label{eq:bgfcirc}
    b=\tilde{\ff}^\circ[\mu]=\tilde{\bb}^\circ[\mu]\quad \mu\text{-a.e.}
\end{equation} (cf.~also Theorem \ref{thm:minsel}), with $\tilde{\ff}^\circ$ being the element of minimal norm of $\bri\frF$ and $\tilde{\bb}^\circ$ of its Lagrangian representation.
By assumption \eqref{eq:Lgstoch} and \cite[Lemma 5.13]{CSS}, the Explicit Euler scheme is solvable for $\frF$.

In the following Proposition \ref{prop:SDFeta}, we prove the equivalence between the measure $\eeta_\tau$ constructed as in Definition \ref{def:intEE} for $\frF$ and the measure on paths defined as the law of a Stochastic Dissipative Flow for $g$.

\begin{proposition}\label{prop:SDFeta} In the setting of Definition \ref{def:SDF}, let $\mu_{0, \tau}:=(X^0_\tau)_\sharp \P$ and let $\frF$ be as in \eqref{eq:F-SDF}. Let $\eeta_\tau\in\prob(\rmC([0,T];\X))$ be as in Definition \ref{def:intEE} for the \PVF $\frF$ and the initial measure $\mu_{0,\tau}$. Then $\eeta_\tau = (X_\tau)_\sharp \P$.
\end{proposition}
\begin{proof}
Let $(M^n_\tau,\Phi_\tau^n)_{0\le n\le \finalstep T\tau}\subset \dom(\frF)\times \frF$ be as in Definition \ref{def:EEscheme} for the \PVF $\frF$. In particular, $\Phi_\tau^n=\frF[M_\tau^n]$. Let $\aalpha_\tau\in\prob(\X^{\finalstep T\tau+1})$ be the multi-step plan associated with $\left(\frF[M_\tau^n]\right)_n$ as in Definition \ref{def:step-plans}. Thus, $\eeta_\tau=G_\sharp\aalpha_\tau\in\prob(\rmC([0,T];\X))$.
By \eqref{eq:SDFjoint}, the result follows if we show that
\[\aalpha_\tau=(X^0_\tau,X^1_\tau,\dots,X^{\finalstep T\tau}_\tau)_\sharp\P.\]

\emph{Claim 1}: $M_\tau^n=(X^n_\tau)_\sharp\P$ for any $0\le n\le\finalstep T\tau$.\\
\emph{Proof of claim 1}: We prove this by induction on $n$. For $n=1$, noting that $(\mu_{0,\tau}\otimes\UU)=(X^0_\tau,V^0)_\sharp\P$ by independence of the random variables, we have
\begin{equation}\label{eq:Mn-joint}
M^1_\tau=(\exp^{\tau})_{\sharp} \frF[\mu_{0,\tau}]=(\exp^{\tau})_{\sharp}(\pi^0,g)_\sharp(\mu_{0,\tau}\otimes\UU)=(\pi^0+\tau g(\pi^0,\pi^1))_\sharp(X^0_\tau,V^0)_\sharp\P={X^1_\tau}_\sharp\P.
\end{equation}
Assume now that the claim is true for $n\ge1$; we need to prove it holds for $n+1$. This proof is analogous, noting that $M^n_\tau\otimes\UU=(X^n_\tau,V^n)_\sharp\P$ again by independence.
\smallskip

\emph{Claim 2}: Denote by $T^n_\tau:=(\sfx,\exp^\tau)_\sharp\frF[M^n_\tau]$ as in Definition \ref{def:step-plans}. Then, its disintegration  $\left\{T^n_{\tau,x}\right\}_{x\in\X}\subset\prob(\X)$ with respect to the projection map $\pi^0$ is given by
\[T^n_{\tau,x}=(x+\tau g(x,\cdot))_\sharp\UU\]
for ${\pi^0}_\sharp T^n_\tau=M^n_\tau$-a.e. $x\in\X$ and for any $0\le n\le\finalstep T\tau$.\\
\emph{Proof of claim 2}: This proof, being straightforward, is omitted.
\smallskip

\emph{Claim 3}: Let $\aalpha_\tau^n$ be the plans defined in Definition \ref{def:step-plans} associated with $\left(\frF[M_\tau^n]\right)_n$, i.e.
\begin{align*}
\aalpha_\tau^1&:=T_\tau^0,\\
\aalpha_\tau^n&:=\int ( \delta_{(x_0,x_1,\dots,x_{n-1})} \otimes T_{\tau,x_{n-1}}^{n-1}) \de \aalpha_\tau^{n-1}(x_0,x_1,\dots,x_{n-1})\in\prob(\X^{n+1}),\quad2\le n\le\finalstep T\tau.
\end{align*}
Then, $\aalpha^n_\tau=(X^0_\tau,X^1_\tau,\dots,X^n_\tau)_\sharp\P$, for any $1\le n\le\finalstep T\tau$.\\
\emph{Proof of claim 3}: We prove the claim by induction argument. First, we have $\alpha^1_\tau=(X^0_\tau,X^1_\tau)_\sharp\P$ and its proof is similar to what is done in \eqref{eq:Mn-joint}. Assume now the claim holds for $n\ge1$ and let us prove it holds for $n+1$. Given any Borel test function $\phi:\X^{n+2}\to\R$, we have
\begin{align*}
\int \phi(x_0,\dots,x_{n+1})\de\aalpha_\tau^{n+1}&=\int \phi(x_0,\dots,x_n,x_{n+1})\de T_{\tau,x_n}^{n}(x_{n+1}) \de \aalpha_\tau^{n}(x_0,\dots,x_{n})\\
&=\int \phi(x_0,\dots,x_n,x_n+\tau g(x_n,u))\de\UU(u) \de \aalpha_\tau^{n}(x_0,\dots,x_{n})\\
&=\int \phi(x_0,\dots,x_n,x_n+\tau g(x_n,u))\de\left[(X^0_\tau,\dots,X^n_\tau,V^n)_\sharp\P\right](x_0,\dots,x_n,u)\\
&=\int \phi(X^0_\tau(\omega),\dots,X^n_\tau(\omega),X^{n+1}_\tau(\omega))\de\P(\omega),
\end{align*}
where we used Claim 2 and the independence of $X^0_\tau,\dots,X^n_\tau,V^n$.
This concludes the proof of the claim and thus the whole proof, since $\aalpha_\tau=\aalpha_\tau^{\finalstep T\tau}$.
\end{proof}

\begin{proof}[Proof of Corollary \ref{cor:6.1}]
Thanks to Proposition \ref{prop:SDFeta}, we can apply Theorem \ref{prop:young} and get that $X_\tau$ converges to ${\mathrm s}_{{\bar X}_\sharp\P} \circ \bar X$ in $L^2(\Omega, \cB, \P; \rmC([0,T]; \X))$, as $\tau\downarrow 0$, where ${\mathrm s}_{{\bar X}_\sharp\P}$ is as in \eqref{eq:s} for the maximal totally dissipative \PVF $\bri{\frF}$. The conclusion follows recalling \eqref{eq:bgfcirc}, see also \cite[Theorem 3.4]{CSS2grande}.
\end{proof}

\subsection{Interaction field}\label{ex:interaction}
Let $f:\X\times\X\to\X$ be a continuous map and let $b_f: \X \times \prob_2(\X) \to \X$ be the barycenter of $f$, defined by
\begin{equation}\label{eq:fbari}
 b_f(x,\mu):=\int_\X f(x,y)\,\d\mu(y), \quad (x, \mu) \in \X \times \prob_2(\X).   
\end{equation}

Given a standard Borel probability space $(\Omega, \cB, \P)$ and $\bar X \in L^2(\Omega, \cB, \P; \X)$, we consider the deterministic ODE
\begin{equation}
    \dot{X}_t=b_f(X_t,{X_t}_\sharp\P),\quad t \in [0,T], \quad X_{t=0}=\bar X.
\end{equation}
We assume that there exists $L>0$ such that
\begin{equation}\label{eq:IntL2}
|f(x, y)|^2 \le L (1+|x|^2+|y|^2) \quad \text{ for every } x, y \in \X,
\end{equation}
and that the vector field $\tilde{f}:\X\times\X\to\X\times\X$ defined by $\tilde{f}(x,y):=\left(f(x,y),\,f(y,x)\right)$ is $\lambda$-dissipative in $\X \times \X$ for some $\lambda\in\R$, i.e.
\begin{equation}\label{eq:dissiptildef}
    \langle\tilde{f}(x_1,y_1)-\tilde{f}(x_0,y_0),\,(x_1,y_1)-(x_0,y_0)\rangle\le\lambda \left(|x_1-x_0|^2+|y_1-y_0|^2\right).
\end{equation}

\begin{remark}
    A particular case in which \eqref{eq:dissiptildef} holds occurs when $f(x,y)=h(x-y)$ and $h:\X\to\X$ is $\lambda$-dissipative and odd.
\end{remark}

Similarly as for the example in Section \ref{ex:sgd1}, we consider the following stochastic scheme driven by the interaction field $f$.

\begin{definition}[IDF]\label{def:IDF} Let $(\Omega, \cB, \P)$ be a standard Borel probability space and $T>0$. Define $J:=\left\{\frac{T}{N}\,:\,N\in\N\setminus\{0\}\right\}$. We say that a family of maps $X_\tau:\Omega\to \rmC([0,T];\X)$, $\tau \in J$, is an \emph{Interaction Dissipative Flow (\IDF)} for $f$ if there exist random variables $(X^n_\tau)_{0 \le n \le N, \tau \in J}, (Y^{n}_\tau)_{0 \le n \le N, \tau \in J} \subset L^2(\Omega; \X)$ such that
\begin{itemize}
    \item $Y^{n}_\tau$, $X^{0}_\tau$ and $(Y^{m}_\tau)_{m<n}$ are independent for every $0\le n\le N$, $\tau \in J$,
    \item $(X^n_\tau)_\sharp \P = (Y^{n}_\tau)_\sharp \P$ for every $0\le n\le N$, $\tau \in J$,
    \item $X^{n+1}_\tau=X^n_\tau+\tau\, f(X^n_\tau,Y^{n}_\tau)$ for every $0\le n\le N-1$, $\tau \in J$,
\end{itemize}
and 
\begin{equation}\label{eq:IDFjoint}
X_\tau=G_\sharp\left(X^0_\tau,X^1_\tau,\dots,X^N_\tau\right) \quad \text{ for every } \tau \in J.
\end{equation}
\end{definition}

\begin{remark}
    Notice that, by construction, for any $\tau\in J$ and $n=0,\dots,N$, we have that $Y^{n}_\tau$ is independent also of $X^n_\tau$.
\end{remark}

We have the following result.
\begin{corollary}\label{cor:intfield} In the setting of Definition \ref{def:IDF}, assume that $(X^0_\tau)_{\tau\in J}$ converges $\P$-a.s.~to some $\bar X \in L^2(\Omega, \cB, \P; \X)$ and that $W_2(({X^0_\tau})_\sharp\P,({\bar X})_\sharp\P)\to 0$ as $\tau\downarrow 0$.
Then $X_\tau$ converges in $L^2(\Omega, \cB, \P; \rmC([0,T]; \X))$ as $\tau \downarrow 0$ to the unique solution of the deterministic ODE
\[\dot{X}_t=b_f(X_t,{X_t}_\sharp\P),\quad t \in [0,T],\quad X_{t=0}=\bar X.\]

\end{corollary}

The result can be proven similarly as done in Section \ref{ex:sgd1} by introducing the following \PVF  $\frF\subset\prob_2(\TX)$,
\begin{equation}\label{eq:Fint1}
\frF[\mu]:=(\pi^0,f)_\sharp(\mu\otimes\mu),\quad\mu\in\prob_2(\X).
\end{equation}
Note that assumption \eqref{eq:IntL2} and \cite[Lemma 5.13]{CSS} give the solvability of the Explicit Euler scheme for $\frF$. 
For $\mu$-a.e. $x\in\X$, its disintegration w.r.t. $\sfx_\sharp\frF=\mu$ is given by
\[\frF_x[\mu]=f(x,\cdot)_\sharp\mu\]
and so, for any $\mu\in\prob_2(\X)$, 
the \PVF $\bri\frF$ is given by
\[\bri\frF[\mu]=\left(\ii_\X, b_f(\cdot,\mu)\right)_\sharp\mu,\quad\mu\in\prob_2(\X).\]
The dissipativity condition on $\tilde{f}$ in \eqref{eq:dissiptildef} gives that
$b_f$ is $\lambda$-dissipative, i.e. for any $\mu_0,\mu_1\in\prob_2(\X)$ and any $\gamma\in\Gamma(\mu_0,\mu_1)$,
\begin{equation}\label{eq:dissipInterf}
    \int_{\X^2}\langle b_f(x_1,\mu_1)-b_f(x_0,\mu_0),\,x_1-x_0\rangle\de\gamma(x_0,x_1)\le\lambda\int |x_1-x_0|^2\de\gamma(x_0,x_1).
\end{equation}
Therefore, $\bri\frF$ is totally $\lambda$-dissipative, thus $\frF$ is unconditionally $\lambda$-dissipative. 
Moreover, the continuity and growth conditions of $f$ imply that $b_f$ satisfies a continuity condition as in the following result.

\begin{lemma}\label{le:skoro} Let $f: \X \times \X \to \X$ be a continuous function satisfying \eqref{eq:IntL2}. Then
\begin{equation}\label{eq:intercont}
    \mu_n \to \mu \text{ in } \prob_2(\X) \quad \Rightarrow \quad (b_f(\cdot, \mu_n))_\sharp \mu_n \to  (b_f(\cdot, \mu))_\sharp \mu\, \text{ weakly in } \prob(\X).
\end{equation}
\end{lemma}
\begin{proof} Since $\mu_n$ converges to $\mu$ in $\prob_2(\X)$, we can find a probability space $(\Omega', \cB', \P')$ and random variables $X_n, X$ in $L^2(\Omega', \cB', \P'; \X)$ such that $(X_n)_\sharp \P' = \mu_n$, $X_\sharp \P'= \mu$ and $X_n \to X$ $\P'$-a.e.~and in $L^2(\Omega', \cB', \P';\X)$ as $n \to + \infty$ (see e.g.~\cite[Theorem B.5]{CSS2grande} or \cite[Proposition 3.23]{CSS2piccolo}). Thus, for every $\varphi \in \rmC_b(\X)$, we may write
\[ \int_\X \varphi \de (b_f(\cdot, \mu_n)_\sharp \mu_n) = \int_{\Omega'} \varphi \left ( \int_{\Omega'} f(X_n(\omega'), X_n(\omega))\de \P'(\omega)  \right ) \de \P'(\omega'). \]
Observe that for $\P'$-a.e.~$\omega \in \Omega'$ and $\P'$-a.e.~$\omega' \in \Omega'$ we have
\[ f(X_n(\omega'), X_n(\omega)) \to f(X(\omega'), X(\omega)) \text{ as } n \to +\infty,\]
since $f$ is continuous. On the other hand condition \eqref{eq:IntL2} gives that
\[ |f(X_n(\omega'), X_n(\omega))|^2 \le L (1+ |X_n(\omega')|^2 + |X_n(\omega)|^2) \quad \text{ for every } \omega, \omega' \in \Omega.\]
Since $X_n(\omega) \to X(\omega)$ for $\P'$-a.e.~$\omega \in \Omega$ and 
\[ \int_\Omega L (1+ |X_n(\omega')|^2 + |X_n(\omega)|^2) \de \P'(\omega) \to \int_\Omega L (1+ |X(\omega')|^2 + |X(\omega)|^2) \de \P'(\omega).\]
By a variant of the dominated convergence theorem (see e.g.~\cite[Theorem 2.8.8, Proposition 4.7.30]{Bogachev}), we deduce that
\[ \int_\Omega f(X_n(\omega'), X_n(\omega)) \de \P'(\omega) \to \int_\Omega f(X(\omega'), X(\omega)) \de \P'(\omega) \quad \text{for $\P'$-a.e.~}\omega' \in \Omega \text{ as } n \to + \infty.  \]
A further application of the dominated converge theorem (recall that $\varphi$ is bounded) gives that 
\[ \int_\Omega \varphi \left ( \int_\Omega f(X_n(\omega'), X_n(\omega))\de \P'(\omega)  \right ) \de \P'(\omega') \to \int_\Omega \varphi \left ( \int_\Omega f(X(\omega'), X(\omega))\de \P'(\omega)  \right ) \de \P'(\omega').\]
In other words 
\[ \int_\X \varphi \de( b_f(\cdot, \mu_n)_\sharp \mu_n) \to \int_\X \varphi \de( b_f(\cdot, \mu_\sharp) \mu). \]
By arbitrareity of $\varphi$ this concludes the proof.
\end{proof}

The continuity in \eqref{eq:intercont} is used as in Section \ref{ex:sgd1} to ensure the maximality of $\bri{\frF}$ so that the analogue of \eqref{eq:bgfcirc} follows.

We can then prove an analogous of Proposition \ref{prop:SDFeta} with $\frF$ as in \eqref{eq:Fint1} and $X_\tau$ as in Definition \ref{def:IDF}. The proof of Corollary \ref{cor:intfield} is then exactly the same of Corollary \ref{cor:6.1}.

\subsection{Nonlocal Stochastic Dissipative Flow}\label{ex:nonlocalSGD}
Assume that $\dim(\X) \ge 2$ and let $\prob_b(\X)$ be the space of probability measures with bounded support. We take a nonlocal vector field $b:\X \times \prob_b(\X) \to\X$ and,  given a standard Borel probability space $(\Omega, \cB, \P)$ and $\bar X \in L^2(\Omega, \cB, \P; \X)$, we consider the deterministic ODE
\begin{equation}
    \dot{X}_t=b(X_t,{X_t}_\sharp\P),\quad t \in [0,T], \quad X_{t=0}=\bar X.
\end{equation}

\begin{example}
The vector field $b$ could be for example a \emph{cylinder} vector field of the form
\[ b(x,\mu) = \sum_{i=1}^N \psi_i \left (\int \varphi_1^i \de \mu, \dots, \int \varphi_{N_i}^i \de \mu\right ) k_i(x), \quad (x, \mu) \in \X \times \prob_2(\X),\]
where $\varphi^i_j : \X \to \R$, $\psi_i : \R^{N_i} \to \R$, and $k_i : \X \to \X$ are smooth and bounded functions.
\end{example}
We assume that $b$ is $\lambda$-dissipative, for some $\lambda\in\R$, i.e.~for any $\mu_0,\mu_1\in\prob_b(\X)$ there exists $\ggamma\in\Gamma_o(\mu_0,\mu_1)$ such that
\begin{equation}\label{eq:dissgnonloc}
    \int_{\X^2}\langle b(x_1,\mu_1)-b(x_0,\mu_0),x_1-x_0\rangle\de\ggamma(x_0,x_1)\le\lambda\int|x_1-x_0|^2\,\d\ggamma(x_0,x_1) = \lambda W_2^2(\mu_0, \mu_1).
\end{equation}

In addition, we assume that $b$ satisfies the following continuity condition:
\begin{enumerate}[label=(C\arabic*)]
\item\label{contb} whenever $(x_n, \mu_n), (x, \mu) \in \X \times \prob_b(\X)$ are such that $x_n \in \supp(\mu_n)$ for every $n \in \N$, the supports of $\mu_n$ are equi-bounded, and $|x-x_n| + W_2(\mu_n, \mu) \to 0$ as $n \to + \infty$, then $b(x_n, \mu_n) \to b(x, \mu)$.
\end{enumerate}
Furthermore, we assume that $b$ is the barycenter of some \PVF $\frF \subset \prob_b(\TX)$; recall that this means that for every $\mu \in \prob_b(\X)$, if we consider the disintegration $\frF[\mu]
=\int \Phi_x\,\d\mu(x)$, then
\begin{equation}\label{eq:isthebari}
    b(x,\mu)=
    \int v\,\d\Phi_x(v)\quad\text{for $\mu$-a.e. $x\in \X$}, \quad \mu \in \prob_b(\X).
\end{equation}
\begin{remark} Condition \eqref{eq:isthebari} is satisfied for example when $b$ can be represented as a stochastic average of a family of Borel vector fields $g:\X \times \prob_b(\X)\times\U\to\X$, where $(\U,\UU)$ is a probability space endowed with a non-atomic probability measure, i.e.
\begin{equation}\label{eq:part}
b(x,\mu)=\int_\U g(x,\mu,u)\,\d\UU(u), \quad  \frF[\mu]:=(\pi^0,g(\cdot,\mu,\cdot))_\sharp(\mu\otimes\UU),\quad (x,\mu) \in \X\times \prob_b(\X).    
\end{equation}
\end{remark}
In the general case, we require the following growth and local boundedness conditions on the chosen \PVF $\frF$ satisfying \eqref{eq:isthebari}:
\begin{enumerate}[label=(F\arabic*)]
\item\label{itemF1} there exists a constant $a\ge0$ such that 
\begin{equation}
    \label{eq:bound-intro2}
    \la v,x\ra\le a(1+|x|^2)
    \quad\text{for $\frF[\mu]$-a.e.~$(x,v)\in \TX$,} \quad \mu \in \prob_b(\X),
\end{equation}
\item\label{itemF2} for every $R>0$ there exists $\rho_R>0$ such that 
\[\supp(\mu)\subset B_R(0)\quad\Rightarrow\quad\supp(\frF[\mu])\subset B_{\rho_R}(0).\] 
\end{enumerate}
We show that the above conditions \ref{itemF1} and \ref{itemF2} imply the solvability of the Explicit Euler scheme for $\frF$.

\begin{lemma}\label{le:ee}
Let $\frF \subset \prob_2(\TX)$ be satisfying \ref{itemF1} and \ref{itemF2}, and let $\bar{\mu} \in \prob_b(\X)$, $T>0$. Then, the Explicit Euler scheme for $\frF$ is approximately solvable at $\bar\mu$ up to time $T$, according to Definition \ref{def:locsolv}.
\end{lemma}
\begin{proof} For every $R>0$, set
 \[ R':= \mathrm{e}^{aT}( R^2+T(1+2a))^{1/2}+1, \quad L:= \rho_{R'}, \quad \bar{\tau} := L^{-2} \wedge T.\]
In order to prove the lemma it is enough to show the following: for every $\mu \in \prob_b(\X)$ with $\supp(\mu) \subset B_R(0)$, for every $0< \tau < \bar{\tau}$ and every $K=1, \dots, \ceil{T/\tau}$, there exist $(M_\tau^n)_{n=0}^K$ and $(R_{n, \tau})_{n=0}^K$ such that 
 \begin{equation}\label{eq:graffa}
 \begin{cases}  
 M_\tau^0 = \mu, \quad R_{0, \tau}=R, \\
 |\frF[M_\tau^n]|_2 \le L & \quad n=0, \dots, K,\\
M_\tau^{n+1}= \exp^{\tau}_\sharp( \frF[M_\tau^n] ),  &\quad n=0, \dots, K-1,\\
\supp(M_\tau^{n}) \subset \overline{B_{R_{n,\tau}}(0)} \subset B_{R'}(0), &\quad n=0, \dots, K, \\
R_{n+1, \tau}^2 = R_{n, \tau}^2(1+2a\tau) + \tau^2 L^2 + 2a\tau, &\quad n=0, \dots, K-1.
 \end{cases}
 \end{equation}
To prove it, we fix $\mu \in \prob_b(X)$ with $\supp(\mu) \subset B_R(0)$ and $0< \tau < \bar{\tau}$ and we proceed by induction on $K$. When $K=1$, we simply define $M_\tau^0, M_\tau^1$ and $R_{0, \tau}, R_{1,\tau}$ as in \eqref{eq:graffa} and we only have to check that
\[ |\frF[M_\tau^n]|_2 \le L, \quad \supp(M_\tau^{n}) \subset \overline{B_{R_{n,\tau}}(0)} \subset B_{R'}(0), \quad n=0,1. \]
We observe that $\supp(M_\tau^0) = \supp(\mu)= B_R(0) \subset \overline{B_{R_0, \tau}(0)} \subset B_{R'}(0)$ by construction. We also have $\supp(M_\tau^1) \subset \overline{\exp^\tau(\supp(\frF[M_\tau^0]))}$ and, if $(x,v) \in \supp(\frF[M_\tau^0])$, it holds
\[ |x+\tau v |^2 \le |x|^2 + \tau^2 |v|^2 + 2a\tau(1+|x|^2) \le R_{0, \tau}^2(1+2a \tau) + \tau^2 L^2 + 2a\tau = R^2_{1, \tau}, \]
where we have used \ref{itemF1} and the fact that $\supp( \frF[M_\tau^0]) \subset B_{\rho_{R'}}(0)$ by \ref{itemF2}. We deduce that $\supp(M_\tau^1) \subset \overline{B_{R_{1, \tau}}(0)}$ and the inequality $R_{1, \tau} < R'$ is trivial using that $(1+2a\tau) \le \mathrm{e}^{2aT}$ and $\tau \le L^{-2} \wedge T$. Finally $|\frF[M_\tau^n]|_2 \le L$ by \ref{itemF2}, for $n=0,1$.

The induction step $K-1 \Rightarrow K$ can be done exactly in the same way, apart from the inequality $R_{K, \tau} < R'$. This inequality follows by applying \cite[Lemma B.2]{CSS} with $x_n:=R_{n, \tau}^2$, $y:=\tau L^2+2a$, $\alpha :=2a$, $N:=K-1$, so that one gets
\[ R_{n, \tau}^2 \le (R^2 + \tau n (\tau L^2 + 2a))\mathrm{e}^{2an \tau} <(R')^2 \quad \text{ for every } 0 \le n \le K,\]
where we used $\tau \le L^{-2}$ and $\tau n \le T$.
\end{proof}

The solvability of the Explicit Euler scheme for $\frF$ produces a family of measures $(M_\tau^n)_n$ as in \eqref{eq:EE}, for every $\tau \in (0,1)$. We can define a general non-local \SDF for $\frF$ as done in Definition \ref{def:SDF}.

\begin{definition}[Nonlocal SDF]\label{def:nnlocSDF} Let $(\Omega, \cB, \P)$ be a standard Borel probability space, $T>0$, $\frF \subset \prob_b(\TX)$ and let $b: \X \times \prob_b(\X) \to \X$ be a non-local vector field satisfying the continuity condition \ref{contb}, \eqref{eq:isthebari}, \eqref{eq:dissgnonloc}, and \ref{itemF1},\ref{itemF2}. Define $J:=\left\{\frac{T}{N}\,:\,N\in\N\setminus\{0\}\right\}$. We say that a family of maps $X_\tau:\Omega\to \rmC([0,T];\X)$, $\tau \in J$, is a \emph{Nonlocal Stochastic Dissipative Flow for $\frF$} if there exist random variables $(X^n_\tau)_{0 \le n \le N, \tau \in J} \subset L^2(\Omega; \X)$ such that
\begin{itemize}
    \item 
    $(X^n_\tau)_n$ is 
    a Markov chain,
    \item the joint law 
of $\big(X^n_\tau,\tau^{-1}(X^{n+1}_\tau-
X^n_\tau)\big)$ 
is $\frF[M^n_\tau]$,
\end{itemize}
and 
\begin{equation}\label{eq:SDFjoint2}
X_\tau=G_\sharp\left(X^0_\tau,X^1_\tau,\dots,X^N_\tau\right)\quad \text{ for every } \tau \in J.
\end{equation}
\end{definition}

\begin{remark}
In the particular case of $\frF$ as in \eqref{eq:part}, then a sequence $(X_\tau^n)_n$ as in Definition \ref{def:nnlocSDF} can be obtained as follows:
there exist $V^k : \Omega \to \U$, $k \in \N$, such that
 \begin{itemize}
     \item $(X^{0}_\tau)_{\tau \in J}$ and $(V^k)_k$ are independent,
     \item $(V^k)_\sharp \P= \UU$ for every $k \in \N$,
     \item $X^{n+1}_\tau=X^n_\tau+\tau\, g(X^n_\tau,({X^n_\tau})_\sharp\P,V^n)$ for every $0\le n\le N-1$, $\tau \in J$.
\end{itemize}    
\end{remark}
We have the following result.
\begin{corollary}\label{cor:nnsdf} In the setting of Definition \ref{def:nnlocSDF}, with $\dim(\X) \ge 2$, assume that $(X^0_\tau)_{\tau\in J}$ converges $\P$-a.s.~to some $\bar X \in L^2(\Omega, \cB, \P; \X)$ such that $\bar{\mu}:= \bar{X}_\sharp \P \in \prob_b(\X)$ and that $W_2(({X^0_\tau})_\sharp\P,\bar \mu)\to 0$ as $\tau\downarrow 0$. Then $X_\tau$ converges in $L^2(\Omega, \cB, \P; \rmC([0,T]; \X))$ as $\tau \downarrow 0$ to the unique solution of the deterministic ODE
\begin{equation}\label{eq:ODE6.3}
\dot{X}_t=b(X_t,{X_t}_\sharp\P),\quad t \in [0,T],\quad X_{t=0}=\bar X.
\end{equation}
\end{corollary}
\begin{proof} Define 
\[\bri{\frF}[\mu]=\left(\ii_\X,\,b(\cdot,\mu)\right)_\sharp\mu, \quad \mu \in \prob_b(\X),\]
and notice that $\bri{\frF}$ is $\lambda$-dissipative by \eqref{eq:dissgnonloc}.
Consider the set $\displaystyle \rmC := \bigcup_{N \in \N} \prob_{\# N}(\X)$, where
\[\prob_{\# N}(\X) := \left \{ \mu \in \prob(X) : \# \supp(\mu) = N \right \}.\]
Notice that $\rmC \subset \prob_2(\X)$ is dense and it is a core in the sense of \cite[Definition 8.1]{CSS2grande}, with $\rmC_N:=\prob_{\# N}(\X)$. We define the \PVF $\Gg$ as the restriction of $\bri{\frF}$ to $\rmC$. \smallskip

\noindent\emph{Claim 1}: for any $N\in\N$, $\Gg$ is demicontinuous over $\rmC_N$ in the following sense:
\begin{equation}\label{eq:Gdemic}
    \mu_n,\mu\in\rmC_N,\,\mu_n \to \mu \text{ in } \prob_2(\X)\quad \Rightarrow \quad \Gg[\mu_n] \to  \Gg[\mu] \text{ in } \prob_2^{sw}(\TX).
\end{equation}
\emph{Proof of claim 1}: since $\mu_n, \mu$ are all concentrated on $N$ distinct points, we have that the supports of $\mu_n$ and $\mu$ are all contained in $B_R(0)$, for some $R>0$. The local boundedness assumption \ref{itemF2} on $\frF$, gives in particular that the supports of $\Gg[\mu_n]=\bri{\frF}[\mu_n]$ are all contained in $B_ \varrho(0)$ for some $\varrho >0$. This implies that 
\[ \sup_n \int |v|^2 \de \Gg[\mu_n](v) < + \infty.\]
To show that $\Gg[\mu_n] \to  \Gg[\mu] \text{ in } \prob_2^{sw}(\TX)$ it remains only to observe that $\Gg[\mu_n] \to  \Gg[\mu]$ in $\prob(\X^s \times \X^w)$. We proceed as in the proof of Lemma \ref{le:skoro}. Let $\varphi \in \rmC_b(\X^s \times \X^w)$ and $(\Omega', \P')$ be a standard Borel probability space. Let $X_n, X \in L^2(\Omega', \P')$ be such that $(X_n)_\sharp \P' = \mu_n$, $X_\sharp \P'= \mu$ and $X_n \to X$ $\P'$-a.e.~in $\Omega'$. Then
\[
    \int_{\X^2} \varphi \de \Gg[\mu_n] = \int_{\Omega'} \varphi(X_n, b(X_n, \mu_n)) \de \P' \to  \int_{\Omega'} \varphi(X, b(X, \mu)) \de \P' = \int_{\X^2} \varphi \de \Gg[\mu],
\]
where we have used the dominated convergence theorem and that $b(X_n, \mu_n) \to b(X, \mu)$ $\P'$-a.e.~in $\Omega'$ in the topology of $\X^s$ (hence also in $\X^w$) due to \ref{contb}. 
\smallskip

\noindent\emph{Claim 2}: $\Gg$ is totally $\lambda$-dissipative; there exists a unique maximal totally $\lambda$-dissipative extension $\hat{\frF}$ of $\Gg$; denoted by $\hat{\frF}^\circ$ the minimal selection of $\hat{\frF}$ (cf.~Theorem \ref{thm:minsel}), we have
\[\hat{\frF}^\circ|_\rmC=\Gg=\bri{\frF}|_\rmC.\]
\emph{Proof of claim 2}: This follows by \cite[Theorem 8.5]{CSS2grande}.
\smallskip

\noindent\emph{Claim 3}: $\bri{\frF} \subset \hat{\frF}$. In particular, $\bri{\frF}$ is totally $\lambda$-dissipative and it has a unique maximal totally $\lambda$-dissipative extension given by $\hat{\frF}$.\\
\emph{Proof of claim 3}:
let $\mu \in \prob_b(\X)$; then we can find $(\mu_n)_n \subset \rmC$ such that the supports of $\mu_n$ are uniformly bounded and $W_2(\mu_n, \mu)\to 0$ as $n \to +\infty$. We can then proceed as in the proof of claim 1 to show that  $\bri{\frF}[\mu_n] \to \bri{\frF}[\mu]$ in $\prob_2^{sw}(\TX)$. Since $\hat{\frF}$ is closed w.r.t.~this convergence by \cite[Proposition 3.16]{CSS2grande}, then $\bri{\frF}[\mu] \in \hat{\frF}[\mu]$. By arbitrariness of $\mu \in \prob_b(\X)$, we conclude. 
\smallskip

\noindent\emph{Claim 4}: Let $\hat{\ff}^\circ\in L^2(\X,\mu;\X)$ be as in Theorem \ref{thm:minsel} for the \PVF $\hat{\frF}$, so that 
$\hat{\frF}^\circ[\mu]= (\ii_\X, \hat{\ff}^\circ[\mu])_\sharp \mu$, for every $\mu\in\dom(\hat{\frF})$. Then, $\hat{\ff}^\circ[\mu]=b(\cdot, \mu)$ for every $\mu \in \prob_b(\X)$.\\
\emph{Proof of claim 4}: let $g: \X \to \X$ be a continuous and bounded function, $\mu \in \prob_b(\X)$ and $\eps\in(0,1)$. Define the measure $\nu_\eps:= (\ii_\X + \eps g)_\sharp \mu \in \prob_b(\X)$ and notice that the supports of $\nu_\eps$ are uniformly bounded w.r.t.~$\eps$ and that $W_2(\nu_\eps, \mu) \to 0$ as $\eps \downarrow 0$. Since both $(\ii_\X, \hat{\ff}^\circ[\mu])_\sharp \mu= \hat{\frF}^\circ[\mu]$ and $(\ii_\X, b(\cdot, \mu))_\sharp \mu = \bri{\frF}[\mu]$ belong to $\hat{\frF}$, we can apply the total $\lambda$-dissipativity of $\hat{\frF}$ (cf.~Definition \ref{def:total-dissipativity}) along the plan $\ggamma :=(\ii_\X, \ii_\X + \eps g)_\sharp \mu \in \Gamma(\mu, \nu_\eps)$ to get
\[ - \int_{\X} \langle \hat{\ff}^\circ[\mu](x)- b(x+\eps g(x), \nu_\eps), g(x) \rangle \de \mu(x) \le \eps \lambda \int |
g(x)|^2 \de \mu.\]
Using the continuity assumption \ref{contb} on $b$ and the local boundedness \ref{itemF2} of $\frF$, we can pass to the limit as $\eps \downarrow 0$ and get
\[ \int_\X \langle \hat{\ff}^\circ[\mu](x)- b(x, \mu), g(x) \rangle \de \mu(x) \ge 0.\]
Being $g$ arbitrary, we must have $\hat{\ff}^\circ[\mu]=b(\cdot, \mu)$, as wanted.
\smallskip

\noindent\emph{Claim 5}: we have ${X_\tau}_\sharp\P=\eeta_\tau$, 
where $\eeta_\tau$ is defined as in Definition \ref{def:intEE} with $(M_\tau^n)_n$ generated by the Explicit Euler scheme for $\frF$.\\
\emph{Proof of claim 5}:
we argue precisely as in the proof of Proposition \ref{prop:SDFeta}. Indeed, by construction we have that $(X_\tau^n, X_\tau^{n+1})_\sharp \P = T_\tau^{n}$ and $(X_\tau^n)_\sharp \P = M_\tau^n$ for every $\tau$ and every $n$, thus showing the first claim in the proof of Proposition \ref{prop:SDFeta}. The third claim in the same proof can be achieved again by induction. Indeed, we observe that the base case $\aalpha_\tau^1=T_\tau^0 = (X^0_\tau, X^1_\tau)_\sharp\P$ has been proven above; while, the induction step can be performed noting that the Markov property of $(X^n_\tau)_n$ implies that the law of $X_\tau^{n+1}$ given $(X_\tau^0, \cdots, X^n_\tau)$ coincides with the law of $X_\tau^{n+1}$ given $X_\tau^n$. Equivalently, we can write
\[ (X_\tau^0, \dots, X_\tau^{n+1})_\sharp \P = \int T^n_{\tau, x_n} \de [(X_\tau^0, \dots, X_\tau^n)_\sharp \P] (x_0, \dots, x_n) = \int T^n_{\tau, x_n} \de \aalpha_\tau^n (x_0, \dots, x_n) = \aalpha_\tau^{n+1}, \]
where we have used the induction hypothesis $(X_\tau^0, \dots, X_\tau^n)_\sharp \P = \aalpha_\tau^n$ and the definition of $\aalpha_\tau^{n+1}$.
\smallskip

\noindent\emph{Claim 6}: $X_\tau$ converges in $L^2(\Omega, \cB, \P; \rmC([0,T]; \X))$ as $\tau \downarrow 0$ to the unique solution of the deterministic ODE in \eqref{eq:ODE6.3}.\\
\emph{Proof of claim 6}:
by the previous claims and Lemma \ref{le:ee}, we can apply Theorem \ref{prop:young} and obtain that $X_\tau$ converges in $L^2(\Omega, \cB, \P; \rmC([0,T]; \X))$ as $\tau \downarrow 0$ to $Z:={\mathrm s}_{\bar\mu} \circ \bar X$, where ${\mathrm s}_{\bar\mu}$ is as in \eqref{eq:s} for the maximal totally $\lambda$-dissipative extension $\hat{\frF}$ of $\bri{\frF}$. This means that $Z$ is the unique solution of
\begin{equation}\label{eq:thelim}
\dot{X}_t = \hat{\ff}^\circ(X_t, (X_t)_\sharp \P), \quad X_{t=0} = \bar{X},    
\end{equation}
where $\hat{\ff}^\circ$ is as in Theorem \ref{thm:minsel} for $\hat{\frF}$. However, by \cite[Theorem 4.2(3)]{CSS2grande} the support of $(Z_t)_\sharp \P$ stays bounded, so that in \eqref{eq:thelim} we can replace $\hat{\ff}^\circ$ with $b$, since, by claim 4, we proved that $\hat{\ff}^\circ[\mu]= b(\cdot, \mu)$ for every $\mu \in \prob_b(\X)$.
\end{proof}

\subsection{ Fully stochastic interaction field}\label{ex:stocint}
We make a variation to the example in Section \ref{ex:interaction} by introducing stochasticity in the interaction field. Let $f: \X \times \X \to \X$ be a continuous map satisfying the dissipativity condition in \eqref{eq:dissiptildef} and consider again the ODE driven by its induced interaction field $b_f$ as in \eqref{eq:fbari}: given a standard Borel probability space $(\Omega, \cB, \P)$ and $\bar X \in L^2(\Omega, \cB, \P; \X)$, we consider the deterministic ODE
\begin{equation}
    \dot{X}_t=b_f(X_t,{X_t}_\sharp\P),\quad t \in [0,T], \quad X_{t=0}=\bar X.
\end{equation}
We consider a particular case of the vector field in Section \ref{ex:interaction}, assuming that $f$ arises as a stochastic superposition, that is, there exists a Borel vector field $h:\X\times \X \times \U\to\X$, where $(\U,\UU)$ is a probability space endowed with a non-atomic probability measure, such that
\[f(x,y)=\int_\U h(x,y, u)\,\d\UU(u), \quad x,y \in \X.\]
We require that there exists $L>0$ such that
\begin{equation}\label{eq:Lhstoch}
\int_\U |h(x,y,u)|^2\,\d\UU(u) \le L(1+|x|^2+|y|^2) \quad \text{ for every } x,y \in \X.
\end{equation}

We state the analogue of Definition \ref{def:IDF} taking into account the dependence of $h$ on $u$.

\begin{definition}[Stochastic IDF]\label{def:stocIDF} Let $(\Omega, \cB, \P)$ be a standard Borel probability space and $T>0$. Define $J:=\left\{\frac{T}{N}\,:\,N\in\N\setminus\{0\}\right\}$. We say that a family of maps $X_\tau:\Omega\to \rmC([0,T];\X)$, $\tau \in J$, is a \emph{Fully Stochastic Interaction Dissipative Flow (fully stochastic \IDF)} for $h$ if there exist random variables $(X^n_\tau)_{0 \le n \le N, \tau \in J}, (Y^{n}_\tau)_{0 \le n \le N, \tau \in J} \subset L^2(\Omega; \X)$, and $V^k : \Omega \to \U$, $k \in \N$ such that
\begin{itemize}
    \item $Y^{n}_\tau$, $X^{0}_\tau$, $(Y^{m}_\tau)_{m<n}$, and $(V^k)_k$ are independent for every $0\le n\le N$, $\tau \in J$,
    \item $(X^n_\tau)_\sharp \P = (Y^{n}_\tau)_\sharp \P$ for every $0\le n\le N$, $\tau \in J$,
    \item $(V^k)_\sharp \P= \UU$ for every $k \in \N$,
    \item $X^{n+1}_\tau=X^n_\tau+\tau\, h(X^n_\tau,Y^{n}_\tau, V^n)$ for every $0\le n\le N-1$, $\tau \in J$,
\end{itemize}
and 
\begin{equation}\label{eq:stocIDFjoint}
X_\tau=G_\sharp\left(X^0_\tau,X^1_\tau,\dots,\right) \quad \text{ for every } \tau \in J.
\end{equation}
\end{definition}

We have the following result.
\begin{corollary}\label{cor:intfieldlast} In the setting of Definition \ref{def:stocIDF}, assume that $(X^0_\tau)_{\tau\in J}$ converges $\P$-a.s.~to some $\bar X \in L^2(\Omega, \cB, \P; \X)$ and that $W_2(({X^0_\tau})_\sharp\P,({\bar X})_\sharp\P)\to 0$ as $\tau\downarrow 0$.
Then $X_\tau$ converges in $L^2(\Omega, \cB, \P; \rmC([0,T]; \X))$ as $\tau \downarrow 0$ to the unique solution of the deterministic ODE
\[\dot{X}_t=b_f(X_t,{X_t}_\sharp\P),\quad t \in [0,T],\quad X_{t=0}=\bar X.\]

\end{corollary}

As for the other examples of Section \ref{sec:examples}, the proof of the above result is based on the following construction.

We define the \PVF $\frF\subset\prob_2(\TX)$ by
\begin{equation}\label{eq:Fintnonl}
\frF[\mu]=(\pi^0,h)_\sharp(\mu\otimes\mu\otimes\UU),\quad\mu\in\prob_2(\X),
\end{equation}
whose barycenter is given by
\[\bri{\frF}[\mu]=\left(\ii_\X,\,b_f(\cdot,\mu)\right)_\sharp\mu.\]

Condition \eqref{eq:dissiptildef} ensures the total $\lambda$-dissipativity of $\bri{\frF}$, while \eqref{eq:Lhstoch} gives the solvability of the Explicit Euler scheme for $\frF$.

We obtain the same result as in Proposition \ref{prop:SDFeta} for the \PVF $\frF$ defined in \eqref{eq:Fintnonl} and with $X_\tau$ defined as in Definition \ref{def:stocIDF}, hence Corollary \ref{cor:intfieldlast}.

\appendix

\section{Technical results}\label{sec:appA}
The following Propositions, of independent interest, are used to prove the strong convergence result in Proposition \ref{prop:tight}. This is determinant to get the main result of the paper in Theorem \ref{thm:main}.

\begin{definition}
Let $(X, d)$ be a metric space and $\mathcal{K} \subset \prob_p(X)$, $p\in[1,+\infty)$. We say that $\mathcal{K}$ has \emph{uniformly integrable $p$-moments} if
\begin{equation}\label{eq:upi}
    \lim_{k \to + \infty} \sup_{\mu \in \mathcal{K}} \int_{\{ x \,:\, d(x,x_0) \ge k\}} d^p(x,x_0) \de \mu(x) =0 \quad \text{ for some (hence for any) } x_0 \in X.
\end{equation}
\end{definition}

\begin{proposition}\label{prop:poussin1} Let $(X, d)$ be a metric space and let $\mathcal{K} \subset \prob_p(X)$, $p\in[1,+\infty)$. 
Assume that there exist $x_0 \in X$ and a Borel measurable function $\varphi:[0,+\infty) \to [0,+\infty)$ such that 
\begin{equation*}
\lim_{r \to + \infty} \frac{\varphi(r)}{r} = + \infty, \quad \sup_{\mu \in \mathcal{K}} \int_X \varphi(d^p(x,x_0)) \de \mu(x) < + \infty.
\end{equation*}
Then $\mathcal{K}$ has uniformly integrable $p$-moments.
\end{proposition}
\begin{proof}
    Let $\eps>0$ and take $R_\eps>0$ such that $\eps\,\varphi(r)>r$ for every $r>R_\eps$. Whenever $k \ge R_{\eps}^{1/p}$, then for any $\mu \in \mathcal{K}$ we have
\begin{align*}
\int_{\{ x \,:\, d(x,x_0) \ge k\}} d^p(x,x_0) \de \mu(x) \le \eps \int_{\{ x\, :\, d(x,x_0) \ge k\}} \varphi(d^p(x,x_0)) \de \mu(x) \le \eps \sup_{\mu \in \mathcal{K}} \int_X \varphi(d^p(x,x_0)) \de \mu(x).
\end{align*}
Passing to the $\sup$ among $\mu \in \mathcal{K}$, to the limit as $k \to + \infty$ and finally to the limit as $\eps \downarrow 0$ proves the sought uniform integrability of the $p$-moments.
\end{proof}

\begin{proposition}\label{prop:poussin2} Let $(X, d)$ be a metric space and let $\mathcal{K} \subset \prob_p(X)$, $p\in[1,+\infty)$. Assume that $\mathcal{K}$ has uniformly integrable $p$-moments. Then, for any $x_0 \in X$, there exists $\varphi \in \rmC^\infty([0,+\infty); [0,+\infty))$ increasing and convex such that 
\begin{equation}\label{eq:thef}
\varphi(0)=0,\quad \lim_{r \to + \infty} \frac{\varphi(r)}{r} = + \infty, \quad \sup_{\mu \in \mathcal{K}} \int_X \varphi(d^p(x,x_0)) \de \mu(x) < + \infty.
\end{equation}
\end{proposition}
\begin{proof}
    The proof is a simple adaptation of the analogous statement for the uniform integrability of a family of functions in $L^1$, see for example \cite[Theorem 4.5.9]{Bogachev}, combined with a regularization argument.\\
Fix $x_0\in X$, by uniform integrability of the $p$-moments, we can find an increasing sequence of natural numbers $C_n \uparrow + \infty$ such that 
\begin{equation}\label{eq:2n}
    \int_{\{ x \,:\, d(x,x_0) \ge C_n^{1/p}\}} d^p(x,x_0) \de \mu(x) \le 2^{-n} \quad \text{ for every } n \in \N, \, \mu \in \mathcal{K}.
\end{equation}
Let us define, for any $k \in \N$ and $\mu \in \mathcal{K}$ the real numbers
\[ \mu_k:= \mu \left ( \left \{ x \in X\, :\, d^p(x,x_0) > k \right \} \right ).\]
We set
\[ \alpha_n:=0 \text{ if } n < C_1 \text{ and } \alpha_n := \max\{ k \in \N\,:\, k \le C_n\} \text{ if } n \ge C_1.\]
We finally define
\[ \tilde{\varphi}(r):= \int_0^r g(s) \de s, \quad g(r):=\sum_{n=0}^{+\infty} \alpha_n\,\nchi_{(n,n+1]}(r), \quad r \ge 0. \]
Clearly $\tilde{\varphi}$ is non-negative, increasing, convex and satisfies $\lim_{r \to + \infty} \tilde{\varphi}(r)/r = + \infty$. Notice also that $\tilde\varphi(r)=0$ for any $r\in[0,1]$. We show that $\tilde\varphi$ also satisfies the last condition in \eqref{eq:thef}. First of all we observe that, for any $n \ge 1$ and any $\mu \in \mathcal{K}$, we have
\begin{align*}
    \int_{\{ x\, :\, d(x,x_0) \ge C_n^{1/p}\}} d^p(x,x_0) \de \mu(x) &= \int_{\{ x \,:\, d^p(x,x_0) \ge C_n\}} d^p(x,x_0) \de \mu(x)\\
    &\ge \sum_{j=C_n}^{+\infty} j \,\mu \left ( \left \{ x \in X \,:\, j <d^p(x,x_0) \le j+1 \right \} \right ) \\
    & \ge \sum_{j=C_n}^{+\infty} (j-C_n+1)\, \mu \left ( \left \{ x \in X \,:\, j <d^p(x,x_0) \le j+1 \right \} \right ) \\
    & = \sum_{k=C_n}^{+\infty} \mu_k.
\end{align*}
We deduce by \eqref{eq:2n}
\begin{equation}
    \sum_{n=1}^{+\infty} \sum_{k=C_n}^{+\infty} \mu_k \le 1 \quad \text{ for every } \mu \in \mathcal{K}.
\end{equation}
Moreover
\begin{align*}
    \int_X \tilde{\varphi}(d^p(x,x_0)) \de \mu(x) &= \int_X \int_0^{d^p(x,x_0)} g(s) \de s \de \mu(x) \\
    &= \sum_{n=0}^{+\infty} \int_{\{ x\,:\, n < d^p(x,x_0) \le n+1\}}  \int_0^{d^p(x,x_0)} g(s) \de s \de \mu(x)\\
    & = \sum_{n=0}^{+\infty} \alpha_n \mu \left ( \left \{ x \in X \,:\, n <d^p(x,x_0) \le n+1 \right \} \right )\\
    & =\sum_{n=1}^{+\infty} \alpha_n \mu \left ( \left \{ x \in X \,:\, n <d^p(x,x_0) \le n+1 \right \} \right )\\
    & \le \sum_{n=1}^{+\infty} \alpha_n \mu_n = \sum_{n=1}^{+\infty} \sum_{k=C_n}^{+\infty} \mu_k \le 1.
\end{align*}
Finally, we define $\varrho \in \rmC_c^\infty(\R)$ as
\[ \varrho(t) := \begin{cases} c_0\exp \left ( \frac{1}{|2x-1|^2-1}\right ) \quad & \text{ if } |2x-1|<1, \\
0 \quad & \text{ if } |2x-1| \ge 1,\end{cases}\]
where $c_0>0$ is the positive constant such that $\int_\R \varrho(x) \de x =1$. Observe that $\supp(\varrho)=[0,1]$, $\varrho(1-t)=\varrho(t)$ for every $t \in \R$ and $\varrho$ is increasing in $[0,1/2]$. We define $\varphi := \varrho \ast \tilde{\varphi}$, where we have extended $\tilde{\varphi}$ to $\R$ by continuity. Clearly (the restriction to $[0,+\infty)$ of ) $\varphi$ is  smooth, convex (since $\tilde{\varphi}$ is convex) and satisfies
\[ \varphi(0)=0, \quad \tilde{\varphi}(t-1) \le \varphi(t) \le \tilde{\varphi}(t) \quad \text{ for every } t \in \R\]
so that \eqref{eq:thef} holds for $\varphi$. It remains to show that $\varphi$ is increasing: we compute
\begin{align*}
    \varphi'(t) &= (\varrho' \ast \tilde{\varphi})(t) = \int_{1/2}^1 \varrho'(r) \tilde{\varphi}(t-r) \de r - \int_0^{1/2} \varrho'(1-r) \tilde{\varphi}(t-r) \de r \\
    &= \int_{1/2}^1 \varrho'(s) \left ( \tilde\varphi(t-s) - \tilde{\varphi}(t-1+s) \right ) \de s \ge 0,
\end{align*}
where the last inequality follows by the fact that $\tilde{\varphi}$ is increasing and that $\varrho$ is decreasing in $[1/2,1]$.
\end{proof}

Applying Propositions \ref{prop:poussin1}, \ref{prop:poussin2}, we provide sufficient conditions to have strong compactness of a set of probability measures over continuous paths.
\begin{proposition}\label{prop:compactness}
    Let $(X, d)$ be a complete and separable metric space, $p\in(1,+\infty)$, $T>0$, and $\mathcal{A}_p: \rmC([0,T]; (X, d)) \to [0,+\infty]$ be the $p$-action functional defined as
\begin{equation}\label{eq:A2}
    \mathcal{A}_p(\gamma) := \begin{cases} \displaystyle\int_0^T |\dot{\gamma}_t|_d^p \de t \quad & \text{ if } \gamma \in \mathrm{AC}^p([0,T]; (X, d)), \\
    + \infty \quad & \text{ else}, \end{cases}
\end{equation}
where $|\dot{\gamma}_t|_d$ is the metric derivative of $\gamma$ at time $t$.
    Let $\mathcal{K} \subset \prob_p(\rmC([0,T]; (X,d))$ be such that 
    \begin{enumerate}
        \item $\displaystyle A:=\sup_{\eeta \in \mathcal{K}}\int \mathcal{A}_p \de \eeta < + \infty$;
        \item $\displaystyle B:=\sup_{\eeta \in \mathcal{K}} \{d(x,x_0) : x \in \supp((\sfe_0)_\sharp \eeta) \} <+\infty$ for some (hence for any) $x_0 \in X$;
        \item $\left\{((\sfe_t)_\sharp\eeta)_{t \in [0,T]}\right\}_{\eeta \in \mathcal{K}} $ is relatively compact in $\rmC([0,T]; \prob_p(X))$.
    \end{enumerate}
Then $\mathcal{K}$ is relatively compact in $\prob_p(\rmC([0,T]; (X,d)))$.
\end{proposition}
\begin{proof}
    Conditions (1),(3) stated above imply (see e.g.~\cite[Theorem 10.4]{abs21}) that $\mathcal{K}$ is uniformly tight in $\prob(\rmC([0,T]; (X,d)))$ so that it is enough (\cite[Proposition 7.1.5]{ags}) to show that $\mathcal{K}$ has uniformly integrable $p$-moments. Since $\{((\sfe_t)_\sharp\eeta)_{t \in [0,T]}\}_{\eeta \in \mathcal{K}}$ is relatively compact in $\rmC([0,T]; \prob_p(X))$, we deduce that 
    \[ \mathcal{K}_0:= \{ (\sfe_t)_\sharp \eeta : t \in [0,T], \, \eeta \in \mathcal{K}\} \subset \prob_p(X)\]
    is relatively compact in $\prob_p(X)$. By Proposition \ref{prop:poussin2}, we deduce that there exists an increasing, convex function $\varphi\in \rmC^\infty([0,+\infty); [0,+\infty))$, such that 
\[
\varphi(0)=0, \quad \lim_{r \to + \infty} \frac{\varphi(r)}{r} = + \infty, \quad C:=\sup_{\mu \in \mathcal{K}_0} \int_X \varphi(d^p(x,x_0)) \de \mu(x) < + \infty.
\]
Denoting by $q$ the conjugate exponent of $p$, we define $\psi:[0,+\infty) \to [0,+\infty)$ as
\[ \psi(r):= \int_0^r  \left (\frac{\varphi(s)}{s} \right )^{1/q} \de s,\quad r \ge 0,\]
where we extended by continuity $\varphi(s)/s$ to $\varphi'(0)$ at $s=0$.
We note that $\psi$ is increasing, $\psi\in\rmC^1([0,+\infty))$ and $\lim_{r \to + \infty} \psi(r)/r= + \infty$. If we prove that 
\begin{equation}\label{eq:unifeta}
    \sup_{\eeta \in \mathcal{K}} \int \psi(d^p_\infty(\gamma, \bar{\gamma})) \de \eeta(\gamma) < + \infty,
\end{equation}
we can conclude the proof by applying Proposition \ref{prop:poussin1}, where $\bar{\gamma} \equiv x_0$. Let us show \eqref{eq:unifeta}; since $\psi$ is locally Lipschitz continuous, if we take $\gamma\in\mathrm{AC}^p([0,T]; X)$ then the composition $[0,T]\ni t\mapsto\psi(d^p_\infty(\gamma_t, \bar{\gamma}))$ is absolutely continuous. Thus we have
\begin{align*}
    \int \psi(d_\infty^p(\gamma, \bar{\gamma})) \de \eeta(\gamma) &= \int \sup_{t \in [0,T]} \psi(d^p(\gamma_t, x_0)) \de \eeta(\gamma)\\
    &\le \int \left ( \psi(d^p(\gamma_0, x_0)) + p\int_0^T \psi'(d^p(\gamma_t, x_0)) d^{p-1}(\gamma_t, x_0) |\dot{\gamma}_t|_d \de t \right ) \de \eeta(\gamma) \\
    & \le \int_X \psi(d^p(x,x_0)) \de ((\sfe_0)_\sharp \eeta)(x) \\
    & \quad + p \left ( \int \int_0^T |\dot{\gamma}_t|_d^p\de t \de \eeta(\gamma)\right )^{1/p} \left ( \int \int_0^T (\psi'(d^p(\gamma_t,x_0)))^q d^p(\gamma_t, x_0) \de t \de \eeta(\gamma)  \right)^{1/q} \\
    & \le \int_X \psi(d^p(x,x_0)) \de ((\sfe_0)_\sharp \eeta)(x) + \\
    & \quad +p \left (\int \mathcal{A}_p \de \eeta \right )^{1/p} \left ( \int_0^T  \int_X \varphi (d^p(x,x_0)) \de ((\sfe_t)_\sharp \eeta)(x)  \de t\right )^{1/q}  \\
    & \le \psi(B^p) + pA^{1/p}(TC)^{1/q},
\end{align*}
since $\psi$ is increasing.
\end{proof}

\section{Sticky particles representation}\label{sec:sticky}
We state and prove the following result providing conditions to ensure uniqueness and sticky behavior of the probabilistic representation associated to a curve of probability measures. This result, despite being interesting by itself, is used to prove Theorem \ref{thm:red} and has been stated, in a simplified form, in Theorem \ref{thm:sticky1}

\begin{theorem}\label{thm:sticky}
Let $N\ge1$, $a_i>0$ with $\sum_{i=1}^N a_i=1$, $x_i\in\X$ such that $x_i\neq x_j$ for $i\neq j$ and set 
\[\bar\mu:=\sum_{i=1}^N a_i\delta_{x_i}.\]
Let $\mu:[0,+\infty)\to\prob_2(\X)$, with $\mu(0)=\bar\mu$, be such that
\begin{equation}\label{pr-mut}
    \#(\supp(\mu_t)) \text{ is finite and non-increasing w.r.t.~} t\ge 0.
\end{equation}
Assume that $\eeta\in \prob(\rmC([0,+\infty);\X))$ is such that $(\mathsf{e}_t)_\sharp \eeta=\mu_t$ for every $t\ge 0$. Then
\[\eeta=\sum_{i=1}^N a_i\delta_{\gamma_i},\]
for curves $\gamma_i\in\rmC([0,+\infty);\X)$ satisfying
\begin{enumerate}[label=(P\arabic*)]
    \item\label{P1} $\gamma_i\neq\gamma_j$, $i\neq j$;
    \item\label{P2} $\gamma_i(0)=x_i$ and $i=1,\dots,N$;
    \item\label{P3} if there exists $s\ge0$, $i,j\in\{1,\dots,N\}$ with $i\neq j$, such that $\gamma_i(s)=\gamma_j(s)$, then $\gamma_i(t) = \gamma_j(t)$ for every $t \ge s$.
\end{enumerate}
In particular, if $\eeta_1,\eeta_2\in \prob(\rmC([0,+\infty);\X))$ are such that $(\mathsf{e}_t)_\sharp \eeta_1=(\mathsf{e}_t)_\sharp \eeta_2=\mu_t$ for every $t\ge 0$, then $\eeta_1=\eeta_2$.
\end{theorem}
\begin{proof}
    We divide the proof into five claims.\\
    \textit{Claim 1}: we have $\mathsf{e}_t(\supp(\eeta)) = \supp(\mu_t)$.\\
    \textit{Proof of claim 1}: by e.g.\cite[Formula (5.2.6)]{ags}, we have that 
    \[ \mathsf{e}_t(\supp(\eeta)) \subseteq \supp(\mu_t) \subseteq \overline{\mathsf{e}_t(\supp(\eeta))}.\]
    In particular, by \eqref{pr-mut} we get that $\mathsf{e}_t(\supp(\eeta))$ is finite, so that it is closed and the conclusion follows.\\
    \quad \\
    \textit{Claim 2}: if $\gamma^1, \gamma^2 \in \supp(\eeta)$ and $\bar{s}\ge0$ are such that $\gamma^1(\bar{s})=\gamma^2(\bar{s})$, then $\gamma^1(t) = \gamma^2(t)$ for every $t \ge \bar{s}$.\\
    \textit{Proof of claim 2}: suppose by contradiction this is not the case and set
    \[ \bar{t}:= \max \left \{  s \ge \bar{s}: \gamma^1(r) = \gamma^2(r) \,\, \text{ for every } \bar{s} \le r \le s \right \} \in [\bar{s},+\infty).\]
    Then, by continuity of $\gamma^1,\gamma^2$, for every $\delta>0$ we can find $t_\delta \in (\bar{t}, \bar{t}+\delta)$ such that $\gamma^1(t_\delta) \ne \gamma^2(t_\delta)$. Let $K := \#(\supp(\mu_{\bar{t}})) \in \{1, \dots, N\}$. If $K=1$, noticing that
    \[\left\{\gamma^1(t_1),\gamma^2(t_1)\right\}\subset\supp(\mu_{t_1}),\]
    then $\#(\supp(\mu_{t_1})) \ge 2 >1$ and we get a contradiction with \eqref{pr-mut}. If $K>1$, then we can find $\tilde{\gamma}_1, \dots, \tilde{\gamma}_{K-1} \in \supp(\eeta)$ such that $\tilde{\gamma}_i(\bar{t}) \ne \tilde{\gamma}_j(\bar{t}) \ne \gamma^1(\bar{t})=\gamma^2(\bar{t})$ for every $i \ne j$. By continuity we can find $\eps_k>0$, $k=1,2$, such that 
    \[ \tilde{\gamma}_i(t) \ne \tilde{\gamma}_j(t) \ne \gamma^k(t) \text{ for every } i \ne j,\,k=1,2,\,\, t \in [\bar{t}, \bar{t}+\eps_k).\]
    Thus, if we set $\eps:= \eps_1 \wedge \eps_2$, we have that 
    \[\left\{\tilde{\gamma}_1(t_\eps),\dots,\tilde{\gamma}_{K-1}(t_\eps),\gamma^1(t_\eps),\gamma^2(t_\eps)\right\}\subset\supp(\mu_{t_\eps}),\]
    hence $\#(\supp(\mu_{t_\eps})) \ge K+1 >K$, a contradiction with \eqref{pr-mut}.
    \\
    \quad \\
    \textit{Claim 3}: we have $\#(\supp(\eeta))=N$.\\
    \textit{Proof of claim 3}: 
    from Claim 1 applied with $t=0$, it follows 
    \begin{equation}\label{eq:suppeta}
    \supp(\eeta) = \bigcup_{i=1}^N \left \{ \gamma \in \supp(\eeta) : \gamma(0)=x_i \right \} := \bigcup_{i=1}^N A_i.
    \end{equation}
    However, Claim 2 applied with $\bar{s}=0$ yields that each $A_i$ is a singleton. This concludes the proof of the claim.\\
    \quad \\
    \textit{Claim 4}: we have
    \[\eeta=\sum_{i=1}^N a_i\delta_{\gamma_i},\]
for curves $\gamma_i\in\rmC([0,+\infty);\X)$ satisfying properties \ref{P1},\ref{P2},\ref{P3}.\\
    \textit{Proof of claim 4}: by Claim 3, there exist $\gamma_i\in\rmC([0,+\infty);\X)$, $i=1,\dots,N$, satisfying \ref{P1} and $\tilde a_i>0$, with $\sum_{i=1}^N\tilde a_i=1$, such that
    \[\eeta=\sum_{i=1}^N \tilde a_i\delta_{\gamma_i}.\]
    From Claim 1 (cf. also \eqref{eq:suppeta}), we get \ref{P2}, i.e. $\gamma_i(0)=x_i$ for any $i=1,\dots,N$; while \ref{P3} comes from Claim 2.
    In particular, since
    \[\sum_{i=1}^N\tilde a_i\delta_{\gamma_i(0)}=(\mathsf{e}_0)_\sharp\eeta=\bar\mu=\sum_{i=1}^N a_i\delta_{x_i},\]
    we deduce that $\tilde a_i=a_i$ for any $i=1,\dots,N$.\\
    \quad \\
    \textit{Claim 5}: if $\eeta_1,\eeta_2\in \prob(\rmC([0,+\infty);\X))$ are such that $(\mathsf{e}_t)_\sharp \eeta_1=(\mathsf{e}_t)_\sharp \eeta_2=\mu_t$ for every $t\ge 0$, then $\eeta_1=\eeta_2$.\\
    \textit{Proof of claim 5}: by Claim 4, we have
\[\eeta_1=\sum_{i=1}^N  a_i\delta_{\gamma^1_i},\qquad\eeta_2=\sum_{i=1}^N  a_i\delta_{\gamma^2_i},\]
with $\gamma^1_i,\gamma^2_i\in\rmC([0,+\infty);\X)$ satisfying properties \ref{P1},\ref{P2},\ref{P3}. Let $\eeta:=\frac{1}{2}(\eeta_1+\eeta_2)$ and notice that $\eeta \in \prob(\rmC([0,+\infty); \X))$ and $(\mathsf{e}_t)_\sharp \eeta= \mu_t $ for every $t \ge 0$. By construction we have
\[ \{ \gamma_i^1, \gamma_i^2\}_{i=1}^N \subset \supp(\eeta);\]
however by Claim 3 we have $\#(\supp(\eeta))=N$, so that it must be that $\gamma_i^1=\gamma_i^2$ for every $i=1, \dots, N$.
\end{proof}

\end{document}